\documentclass[12pt]{article}
\usepackage{amsfonts,amssymb,amsmath,url}
\usepackage[all,arc]{xy}

\usepackage{hyperref}

\newtheorem{theorem}{Theorem}[section]
\newtheorem{prop}[theorem]{Proposition}
\newtheorem{lemma}[theorem]{Lemma}
\newtheorem{cor}[theorem]{Corollary}

\newtheorem{qn}[theorem]{Question}
\newenvironment{question}{\qn\rm}{\endqn}
\newtheorem{exam}[theorem]{Example}
\newenvironment{example}{\exam\rm}{\endexam}
\newtheorem{unnumber}{}

\newtheorem{prepf}[unnumber]{Proof}
\newenvironment{proof}{\prepf\rm}{\endprepf}
\newtheorem{prerk}[theorem]{Remark}
\newenvironment{remark}{\prerk\rm}{\endprerk}

\newcommand{\Aut}{\mathop{\mathrm{Aut}}}
\newcommand{\rk}{\mathop{\mathrm{rk}}}
\newcommand{\Inn}{\mathop{\mathrm{Inn}}}
\newcommand{\Car}{\mathop{\mathrm{Car}}\limits}
\newcommand{\Dir}{\mathop{\mathrm{Dir}}\limits}
\newcommand{\GL}{\mathop{\mathrm{GL}}}
\renewcommand{\wp}{\mathbin{\bar\wr}}

\newcommand{\qed}{\hfill$\Box$}

\newcommand{\sdir}{\rtimes}
\newcommand{\nor}{\unlhd}

\newcommand{\Z}{\mathbb{Z}}
\newcommand{\N}{\mathbb{N}}

\def\nor{\trianglelefteq\,}

\begin{document}
\title{Integrals of groups II}
\author{Jo\~ao Ara\'ujo\footnote{Universidade Aberta, 
R. Escola Politecnica 147, 1269-001, Lisboa, Portugal and
CEMAT-Ci\^encias, Faculdade de Ci\^encias, Universidade de Lisboa, 1749-016, Portugal
\texttt{\href{mailto:joao.araujo@uab.pt}{joao.araujo@uab.pt}}}, 
Peter J. Cameron\footnote{School of Mathematics and Statistics, University
of St.~Andrews, UK and CEMAT-Ci\^encias, Faculdade de Ci\^encias, Universidade de Lisboa, 1749-016, Portugal, \texttt{\href{mailto:pjc20@st-andrews.ac.uk}{pjc20@st-andrews.ac.uk}}},
Carlo Casolo, \\
Francesco Matucci\footnote{Dipartimento di Matematica e Applicazioni, Universit\`a di Milano--Bicocca, Italy, \texttt{\href{mailto:francesco.matucci@unimib.it}{francesco.matucci@unimib.it}}
}, and Claudio Quadrelli\footnote{Dipartimento di Scienza e Alta Tecnologia, Universit\`a dell'Insubria, Italy, \texttt{\href{mailto:claudio.quadrelli@uninsubria.it}{claudio.quadrelli@uninsubria.it}}
}}
\date{April 2022}
\maketitle

\begin{quote}\small
\textit{In April 2018, Carlo Casolo sent the other authors detailed answers
to some of the questions in the first version of the paper~\cite{accm}, and we
immediately invited him to join us. He was very dedicated and curious about
integrals and inverse group theory problems. In fact, the current paper is in
large part Carlo's work, together with the fruits of a meeting in Florence in
February 2020. Carlo passed away not long after. He was very generous and kind
to all of us and is sorely missed. We dedicate this paper to his memory.}
\end{quote}

\begin{abstract}
An \emph{integral} of a group $G$ is a group $H$ whose derived group
(commutator subgroup) is isomorphic to $G$. This paper continues the investigation
on integrals of groups started in the work \cite{accm}.
We study:
\begin{itemize}\itemsep0pt
\item A sufficient condition for a bound on the order of an integral for a finite integrable group (Theorem~\ref{t:suf}) and a necessary condition for a group to be integrable (Theorem~\ref{t:nec}).
\item The existence of integrals that are $p$-groups for abelian $p$-groups, and of nilpotent integrals for all abelian groups (Theorem~\ref{t:nilpotent-integral}).
\item Integrals of (finite or infinite) abelian groups, including nilpotent integrals, groups with finite index in some integral, periodic groups, torsion-free groups and finitely generated groups (Section~\ref{sec:integrabelian}).
\item The variety of integrals of groups from a given variety, varieties of
integrable groups and classes of groups whose integrals (when they exist)
still belong to such a class (Sections~\ref{sec:varieties} and \ref{sec:self-integrating}).
\item Integrals of profinite groups and a characterization for integrability for
finitely generated profinite centreless groups (Section~\ref{ss:profabs}).
\item Integrals of Cartesian products, which are then used to construct
examples of integrable profinite groups without a profinite integral (Section~\ref{sec:products}).
\end{itemize}
We end the paper with a number of open problems.
\end{abstract}

\section{Introduction}
\label{s:intro}

In our recent paper~\cite{accm}, we defined an integral of a group $G$ to be
a group $H$ whose derived group is isomorphic to $G$, and called a group $G$
integrable if it has an integral.

We traced this idea back to a paper of Bernhard Neumann~\cite{neumann} in 1956,
but it is much older. In 1913, Burnside published a paper~\cite{burnside} in
which he considered the question (in our language) of whether a given $p$-group
has an integral which is a $p$-group. We call such a group $p$-integrable,
and devote a section to such groups below. Every abelian $p$-group is
$p$-integrable, but it follows from Burnside's results that there are 
$p$-groups which are integrable but not $p$-integrable (the smallest being
the quaternion group of order~$8$).

We treat a number of further topics. The longest part (represented by Sections~\ref{sec:p-integrals}
and~\ref{sec:integrabelian})
of the paper 
considers integrals of infinite abelian groups in some detail. We also
examine profinite groups in Section~\ref{sec:profinite}
where we show that if a profinite group $G$ is
integrable, and if either $G$ is finitely generated (as profinite group) or
$G$ has finite index in some integral, then $G$ has a profinite integral. 
We also give a characterization for integrability in the case of
finitely generated profinite centreless groups and then provide examples
of integrable profinite groups without a profinite integral. 

An important question left open in \cite{accm} is to find an explicit bound
in terms of $G$ for the order of some integral of the finite group $G$ (if it
has one). Such a bound would give us an algorithm for testing integrability
of a finite group. We were able to give bounds in some special cases, 
including abelian groups and centreless groups. In this paper, we push the
analysis further in Section~\ref{s:bounding}. We show that, to bound the order
of some integral of $G$, it suffices to bound the exponent of the centre of
some integral of $G$ in terms of $G$.

We observe that Bettina Eick has a characterization for groups that are Frattini
subgroups of other groups and look for a similar characterization, obtaining a necessary
condition for integrability in Section~\ref{sec:necessary}.

In Sections~\ref{sec:varieties} and ~\ref{sec:self-integrating}
we work on integrals of groups from a given variety showing that such class forms a variety and tackling
the question of whether it is finitely based. We also look at whether there are varieties of integrable groups
beyond that of abelian groups and also study classes of groups so that, whenever we have an integrable group $G$ in such a class
$\mathcal{C}$, then $G$ has an
integral in $\mathcal{C}$.

In Section~\ref{sec:questions} we discuss the solution, by Efthymios Sofos, to
Question 10.1 from our work~\cite{accm}, and give some more open questions.

In this paper, we use $\Car$ and $\Dir$ for the (unrestricted) Cartesian
product and the direct sum, respectively, of a family of groups.


\section{Bounding the order of an integral}
\label{s:bounding}

A problem left open in the first paper~\cite{accm} is to find a bound for the
integral of an integrable finite group $G$ in terms of $G$. If such a bound can
be found, then we have at least a computable test for the integrability of $G$
(though not a very efficient test): compute all groups of order divisible
by $|G|$ up to the bound, and decide for each group whether its derived group
is isomorphic to $G$.

We have now an argument that reduces this problem to the problem of finding
a bound for the exponent of $Z(H)$ (for some integral $H$ of $G$) in terms of
$G$. It is still open how to find such a bound, if it exists.

\begin{theorem}
\label{thm:integral-bound}
Suppose there is a function $F$ from finite groups to natural
numbers such that, if $G$ is an integrable finite group, then $F(G)$ is
a bound for the exponent of the centre of some integral $H$ of $G$. Then there
is a function $F^*$ from finite groups to natural numbers suth that, if $G$ is
an integrable finite group, then $G$ has an integral of order at most $F^*(G)$.
\label{t:suf}
\end{theorem}

\begin{proof}
Let $G$ be a finite group, $H$ an integral of $G$. We can assume
$H$ to be finite by \cite[Theorem~2.2]{accm}.

The proof proceeds by three reductions:

\paragraph{Step 1} 
Let $K=C_H(G)$. Then $H/K\le\Aut(G)$. So it suffices to bound
$|K|$.

\paragraph{Step 2} It suffices to bound $|Z(H)|$.

To see this, let $h_1,\ldots,h_t$ generate $H$. We know that $t$ is bounded
in terms of $G$: $t\le2\mu(G)$, where $\mu(G)$ is the maximal size of a 
minimal generating set for $G$. This is because $G$ is generated by 
commutators $[h,h']$ for $h,h'\in H$; choose a minimal set of commutators
which generate $G$, and replace $H$ by the subgroup generated by the elements
appearing in those commutators.

Next, for $i=1,\ldots,t$, define $\phi_i:K\to Z(G)$ by
\[\phi_i:x\mapsto[x,h_i].\]
Note that $[x,h_i]\in K\cap G=Z(G)$.

Take any $x,y\in K$. Then $[y,h_i]\in G$, so this element commutes with $y$.
So a standard commutator
identity shows that $[xy,h_i]=[x,h_i][y,h_i]$, that is, $\phi_i$ is a
homomorphism. Its kernel is $C_K(h_i)$ and its image is contained in $Z(G)$.
So $|K/C_K(h_i)|\le|Z(G)|$.

It follows that
\[|K/\bigcap_{i=1}^tC_K(h_i)|\le|Z(G)|^t.\]

Now $\bigcap_{i=1}^tC_K(h_i)=C_K(H)$ (since $H=\langle h_1,\ldots,h_t\rangle$)
and $C_K(H)=K\cap Z(H)=Z(H)$, because $Z(H)\le C_H(G)=K$. Thus we conclude
that $|K/Z(H)|\le|Z(G)|^t \le|Z(G)|^{\mu(G)/2}$, proving our claim.

\paragraph{Step 3} It suffices to bound the exponent of $Z(H)$ in terms of $G$.

We show that, without loss of generality, $\rk(Z(H))\le\rk(Z(G))$,
where $\rk(A)$ is the rank of the abelian group $A$, (the minimal number
of generators). It follows that 
\[|Z(H)|\le(\exp(Z(H))^{\rk(Z(G))},\]
and the step is complete.

Suppose that $\rk(Z(H))>\rk(Z(G))$. Then there is a subgroup $N$ of
$Z(H)$ with $Z(G)\cap N=1$. Then
\[N\cap G\le N\cap(Z(H)\cap G)\le N\cap Z(G)=1.\]
So
\[(H/N)'=H'N/N\cong G/(G\cap N)\cong G,\]
so $H/N$ is a smaller integral of $G$ and we can use that instead. This
reduction terminates with $\rk(Z(H))\le\rk(Z(G))$. \qed
\end{proof}

At this point, we hit an obstruction:

\begin{example}
For every $n\ge3$, the group $C_2$ has an integral
\[G_n=\langle a,b\mid a^{2^{n-1}}=b^2=1,b^{-1}ab=a^{2^{n-2}+1}\rangle\]
of order $2^n$, with $Z(G_n)$ cyclic of order $2^{n-2}$. Every proper subgroup
or factor group of the group $G_n$ is abelian, so it is not at all clear how
we could ``reduce'' it to a group with smaller cyclic centre, although clearly such groups do exist.
\end{example}


\section{Towards a characterization of integrable groups}
\label{sec:necessary}

Bettina Eick~\cite{eick} proved the following remarkable theorem. Here
$\Phi(G)$, $\Aut(G)$ and $\Inn(G)$ denote the Frattini subgroup, automorphism
group, and inner automorphism group of the group $G$.

\begin{theorem}[Eick~\cite{eick}]
The finite group $G$ is the Frattini subgroup of some group $H$ if and only if
$\Inn(G)\le\Phi(\Aut(G))$.
\end{theorem}

This gives a test, involving looking only at $G$, to decide whether a group is
a Frattini subgroup.

However, this is false if we replace ``Frattini subgroup'' by ``derived
subgroup''. We showed in \cite{accm} that the non-abelian group $G$ of order
$27$ and exponent $9$ is not integrable; but its inner automorphism group is
contained in the derived group of its automorphism group. (The automorphism
group of $G$ has order~$54$; its derived group has order~$27$, so is a 
normal Sylow-$3$-subgroup and contains all $3$-subgroups of $\Aut(G)$,
including $\Inn(G)$ which has order $9$.)

An analogue of Eick's result for the derived group holds for various classes of groups, including abelian groups and perfect groups. Moreover, the test works in general one way round:

\begin{theorem}
If the group $G$ is integrable, then $\Inn(G)\le\Aut(G)'$; indeed
$\Inn(G)$ has an integral within $\Aut(G)$.
\label{t:nec}
\end{theorem}

\begin{proof}
Let $H$ be an integral of $G$, and $K=C_H(G)$; then $H/K$ embeds in $\Aut(G)$,
and $K\cap G=Z(G)$. We have
\[(H/K)'=H'K/K=GK/K\cong G/(G\cap K)=G/Z(G)\cong\Inn(G),\]
so $H/K$ is an integral of $\Inn(G)$ and is contained in $\Aut(G)$. Moreover,
$\Inn(G)=(H/K)'\le\Aut(G)'$ since $H/K\le\Aut(G)$. \qed
\end{proof}


\section{$p$-integrals}
\label{sec:p-integrals}

We say that a $p$-group (finite or infinite) is \emph{$p$-integrable}
if it has an integral which is a $p$-group.

As noted in \cite{accm}, Guralnick~\cite{guralnick} observed that an abelian
group $A$ of order $n$ has an integral of order $2n^2$, namely $A\wr C_2$. So
any finite abelian $2$-group has a $2$-integral.

More generally, any abelian $p$-group has a $p$-integral.
We give a more general argument which will be used in the next section also.

\begin{theorem}
\begin{enumerate}
\item[(a)]\label{7.a} Every abelian group has an integral which is nilpotent of class~$2$.
\item[(b)]\label{7.b} Every finite abelian $p$-group has a $p$-integral which is finite
and nilpotent of class~$2$.
\item[(c)]\label{7.c} Every abelian $p$-group has a $p$-integral which is a nilpotent
$p$-group of class~$2$.
\item[(d)]\label{7.d} Every periodic abelian group has an integral that is periodic and nilpotent of class $2$.
\end{enumerate}
\label{t:nilpotent-integral}
\end{theorem}

\begin{proof}
Let $A$ be an abelian group. Recall that a group is nilpotent of class $2$ if and only if its derived group is nontrivial and central.

Suppose first that $A$ is the additive group of a ring $R$ with identity.
Then, as remarked in \cite{accm}, the group $G$ of upper unitriangular
$3\times3$ matrices over $R$ is nilpotent of class~$2$ and satisfies
$G'\cong A$.

Not every abelian group is the additive group of a ring with identity. For
if $A$ is the additive group of $R$, then the exponent of $A$ is equal to
the additive order of the identity of $R$; so a torsion abelian group of
unbounded exponent will fail this property. However, two classes of groups
which do have the property are
\begin{itemize}
\item Finitely generated abelian groups; such a group is a finite direct
sum of cyclic groups, and a cyclic group is the additive group of the ring
of integers or of integers mod~$n$, according as its exponent is infinite
or finite. (Part (b) of the theorem follows from this, since if $A$ is
finite then $|G|=|A|^3$.)
\item Free abelian groups. For let $A$ be a free abelian group. By the
previous case, we can assume that $A$ is not finitely generated. If its
rank is the cardinal number $\lambda$, then it is the additive group of the
ring of polynomials in $\lambda$ indeterminates over $\mathbb{Z}$.
\end{itemize}

Now, let $A = F/R$ be an abelian group, where $F$ is free abelian. Let $T$ be
an integral of $F$, with $F = T' \le Z(T)$. Then $R\le Z(T)$, so $R$ is normal
in $T$; setting $H = T/R$ we have $H' = F/R = A$.

The proof of (c) requires a little more care. 
Let $A$ be an abelian $p$-group, and write $A=F/R$, where $F$ is free abelian,
say $F=\displaystyle{\Dir_{i\in\Lambda}\langle f_i\rangle}$. There is an
epimorphism $\theta:F\to A$. Let the order of $f_i\theta$ be $p^{r_i}$. Now
let $G=\displaystyle{\Dir_{i\in\Lambda}C_i}$, where $C_i=\langle g_i\rangle$
is a cyclic group of order $p^{r_i}$, and let $\phi$ be the epimorphism from 
$F$ to $G$ defined by $f_i\phi=g_i$. We show that $\theta$ factors through
$\phi$. Let $f=\sum n_if_i$ (a finite sum) belong to the kernel of $\phi$.
Then $\sum n_ig_i=0$, so $p^{r_i}\mid n_i$ for all $i$; but this implies that
$f\theta=\sum n_i(f_i\theta)=0$, so $\sum n_if_i\in\ker(\theta)$. In other
words, $\ker(\phi)\le\ker(\theta)$.

Thus, there is an epimorphism $\psi:G\to A$ such that $\phi\psi=\theta$.
(For $g\in G$, define $g\psi=f\theta$, where $f$ is a preimage of $g$ under
$\phi$; the condition on kernels shows that this is well-defined.) So we
have $A\cong G/S$ for some $S$.

For each $i\in\Lambda$, let $D_i$ be a group isomorphic to the group of
upper unitriangular matrices of dimension $3$ over $\mathbb{Z}/p^{r_i} \mathbb{Z}$;
its centre, which is equal to its derived group, is cyclic of order $p^{r_i}$,
and we identify this group with $\langle f_i\theta \rangle$. Let
$H=\displaystyle{\Dir_{i\in\Lambda}D_i}$. Then $H$ is a $p$-group, and $H'=G$.
Also $S\le Z(H)$, so $S\lhd H$; and $S\le H'$, so
\[(H/S)'=H'/S=G/S\cong A,\]
and we are done.

Part (d) is a consequence of part (c), since a periodic abelian group is a
direct sum of $p$-groups. \qed
\end{proof}

For non-abelian $p$-groups, some are $p$-integrable, for there are $p$-groups
of arbitrarily large derived length. But of the groups of order $8$,
the three abelian groups are $2$-integrable; the dihedral group
is not integrable; and the quaternion group is integrable (it has an integral
$\mathrm{SL}(2,3)$ of order $24$) but  not $2$-integrable. Indeed, the
following two theorems were proved by Burnside~\cite{burnside}. Either of
them deals with $Q_8$.

\begin{theorem}[Burnside~\cite{burnside}]
\begin{enumerate}
\item[(a)] A non-abelian $p$-group with cyclic centre is not $p$-integrable.
\item[(b)] A non-abelian  $p$-group whose derived group has index $p^2$ is not
$p$-integrable.
\label{t:not-p-int}
\end{enumerate}
\end{theorem}

Another open problem is to find the smallest $p$-integral of a given
$p$-integrable group. For the three abelian groups of order $8$, the smallest
$2$-integral of the cyclic group has order $32$, and for the other two the
smallest $2$-integral has order $64$.

\medskip

We consider further the question of the smallest $2$-integral of an elementary
abelian $2$-group, since this is relevant for the discussion of integrals of
infinite abelian groups in the next section. Any elementary abelian $2$-group
of order $n>4$ has an integral of order $n^2$, namely a \emph{Suzuki 
$2$-group}, see Higman~\cite{higman}. However, we can do substantially better.

Let $A$ be an elementary abelian $2$-group of order $2^n$.  We start with an
example. Logarithms are in base $2$.

\begin{example}
\label{thm:p-integral-example}
Suppose that $A$ is elementary abelian of order $2^n$, and let $k$ be a positive integer with $k>\log n$. Let $H$ be an abelian $2$-group of order $2^k$ and consider the standard wreath product
\[W = C_2\wr H = B\rtimes H,\]
where $C_2$ is a cyclic group of order $2$, and $B$ is the base group of the product. We have that $W'\le B$ is an elementary abelian $2$-group of index $2$ in $B$; hence
\[|W'| = 2^{|H|-1} = 2^{2^k-1}\ge 2^n.\]
Since $W$ is nilpotent, it admits a normal subgroup $N \le W'$, with $|W'/N| = 2^n$; we let $G = W/N$. Then
$G' = W'/N\cong A$, so $G$ is an integral of $A$, and 
\[ |G| = 2^{n+1}2^k.\]
Now, we may well take $k = \lfloor\log n\rfloor+1$, and obtain
\[|G| \le |A|2^{\lfloor\log n\rfloor+2}.\]
Thus, $f(n) \le n + \lfloor\log n\rfloor+2$,
where $f(n)=\log F(A)$ is the function defined in Theorem~\ref{thm:integral-bound}.
\end{example}

Observe that the inequality above implies
\begin{equation}
|G/A| \le 4\log |A|.
\end{equation}

\medskip
Now, we aim at a lower bound for $f(n)$. We require the following results \cite[Lemmas 4.1 and 4.2]{accm}.

\begin{lemma}\label{lemma1} Let $H$ be a $2$-group acting by automorphisms on the finite elementary abelian $2$-group $A$, then
\[|A/[A,H]|\ge |A| ^{1/|H|}.\]
\end{lemma}

\begin{theorem}\label{lemma2-red} Let $A$ be a finite elementary abelian $2$-group, and $G$ a $2$-group such that $G'=A$; write $H = G/A$, then
\begin{equation}
|H|\log^2 |H| \ge 2\log |A|.
\end{equation}
\end{theorem}

\begin{cor}\label{corol-red} Let $A$ be a $2$-group, with $A/A^2$ infinite, and let $G$ be a $2$-group such that $A = G'$. Then $G/A$ is infinite.
\end{cor}

\begin{proof} As $A^2$ is characteristic in $A$ and $(G/A^2)'= A/A^2$, we may well assume $A^2=1$, so that $A$ is an infinite elementary abelian $2$-group. 

Suppose, by contradiction, that $G/A$ is finite. Given any finite index subgroup
$H \le A$, its normal core $H_G$ in $G$ has finite index too, so by taking finite index subgroups of $A$ of increasingly larger order, 
we find subgroups $N\le A$ with $N\unlhd G$ and $A/N$ finite and arbitrarily large. But $A/N$ is the derived subgroup of $G/N$, and this contradicts Lemma \ref{lemma2-red}. \qed
\end{proof}

According to \textsf{GAP}~\cite{gap}, the order of the smallest $p$-integral
of an elementary abelian group $A$ is $8|A|$ for $|A|=2^k$ with $2\le k\le 5$,
and $9|A|$ for $|A|=3^k$ with $1\le k\le 4$. However, as we have seen above,
such bounds cannot hold in general. A small open problem: find the largest
value of $k$ for which one of them holds. 
For example, if $p=2$, the above
result shows that, if the elementary abelian $2$-group of order $2^k$ has
index $8$ in its smallest $2$-integral, then $2^3 \cdot 3^2\ge2k$, so $k\le36$.
What is the exact value?

We will see in Corollary \ref{p:int-p-gp} an estimate of the smallest order of
a $p$-integral of a finite abelian $p$-group when $p$ is an odd prime number.


\section{Integrals of abelian groups}
\label{sec:integrabelian}

We know that every abelian group has an integral. Here, we are concerned with the existence of integrals of an abelian group that are in some sense `close'  to the group.

\subsection{Nilpotent integrals}

We have seen that every abelian group has an integral which is nilpotent of
class~$2$, in Theorem~\ref{t:nilpotent-integral} above.
Let us add some observations in the finite case. 

\begin{lemma}\label{lemma00} Let $A$ be a finite abelian $p$-group of rank $d$ and exponent $p^n$. Then there exists a finite $p$-group $G$ such that $A = G'$ and $|G| = |A|p^{n+d}$. If $p=2$, $|G| = |A|2^{d+1}$.
\end{lemma}

\begin{proof} Let $A = \langle x_1\rangle\times \ldots \times \langle x_d\rangle$, with $|x_1|=p^n$.  
Then consider a group $N = \langle y_1\rangle\times \ldots \times \langle y_d\rangle$, where, for every $i = 1, \ldots, d$, $y_i^p = x_i$; let $\alpha $ be the automorphism of $N$ defined by $y_i^{\alpha} = y_i^{p+1}$; thus, $|\alpha| = p^n$.
(A simple induction shows that $(1+p)^{p^n}\equiv1\pmod{p^{n+1}}$.) 
Finally, set $G = N\sdir\langle \alpha\rangle$. Then $|G| = |N||\alpha| = |A|p^dp^n$, and
\[ G' = [N, \alpha] = \{[y_i,\alpha] \mid i = 1, \ldots, d\} = \{y_i^p \mid i = 1, \ldots, d\}  =  A.\]
If $p=2$, the automorphism $\alpha$ above may be taken to be the inverse map, yielding $|G| = |A|2^{d+1}$. \qed
\end{proof}

We do not claim that for a single abelian $p$-group $A$ this construction yields a nilpotent integral of smallest possible order: for instance, if $A$ is the direct sum of $p-1$ cyclic groups of order $p^n$, then $A$ is isomorphic to the derived subgroup of $H = C_{p^n}\wr C_p$, and $|H| = |A|p^{n+1}$.  It however provides a smallest nilpotent integral when $A$ is cyclic and $p\ne 2$.

\begin{lemma}\label{lemma01} Let $p$ be an odd prime and $G$ a finite non-abelian $p$-group such that $G'$ is cyclic of order $p^n$. Then $|G| \ge p^{2n+1}$.
\end{lemma}

\begin{proof} Let $G$ be a finite $p$-group such that $A= G'$ is cyclic of order $p^n$, with $n\ge 1$. Now $G/C_G(A)$ is cyclic, so setting $C = C_G(A)$, there exists $y\in G$ such that $G = C\langle y\rangle$.

Suppose first that $C'= A$. Then, since $A$ is cyclic of prime power order, $A = \langle [a,b]\rangle$ for some $a,b\in C$, and $[a,b,b]=1$; hence
$1\ne [a,b]^{p^{n-1}} = [a,b^{p^{n-1}}]$,
and so $b^{p^{n-1}}$ does not belong to $C_C(a) \ge A\langle a\rangle > A$. 
So the elements $1,b,b^2,\ldots,b^{p^n-1}$ belong to distinct
cosets of $A\langle a\rangle$.
This yields $|C| \ge p^{n} |A\langle a\rangle| \ge p^{2n+1}$. 

We may then assume $C' < A$. Since $A = (C\langle y\rangle)' = C'[C,\langle y\rangle]$ is cyclic, we deduce that there exists $x\in C$ such that $A = \langle [x,y]\rangle$. Clearly, we may also suppose $G = \langle x, y \rangle$. In this setting we prove, by induction on $n$, that  $[x,y^{p^{n-1}}] \ne 1$. The case $[x,y,y]=1$ has already been proved above, and includes the case $n=1$. Thus, let $n\ge 2$ and $K = [A,y,y] = \langle [x,y,y,y]\rangle$; observe that $|K|\le p^{n-2}$, hence it is properly contained in $\langle [x,y]^p\rangle$. Now,
\[[x,y^p] = x^{y^p-1} = x^{(y-1)(1+ y + \cdots+  y^{p-1})}= [x,y]^{1+ y + \cdots+  y^{p-1}},\]
thus, by standard commutator calculus,
\[\begin{array}{rcl}
[x,y^p] &=&[x, y]^{p}[x, y]^{(y-1)+\left(y^{2}-1\right)+\cdots+\left(y^{p-1}-1\right)} \\ 
           &=&[x, y]^{p}[x, y, y]^{1+(1+y)+\left(1+y+y^{2}\right)+\cdots+\left(1+y+y^{2}+\cdots+y^{p-2}\right)} \\ 
           &\equiv&  [x,y]^p[x,y,y]^{1+ 2 + \cdots+ (p-1)} \pmod K\\
           &=&  [x,y]^p[x,y,y]^{\frac{p(p+1)}{2}},
\end{array}
\] (as $[x,y,y]^y\equiv [x,y,y] \pmod K$).
Since $p$ is odd and $A=\langle [x,y]\rangle$ is cyclic, we deduce 
\begin{equation}\label{eq00}
\langle [x,y^p] \rangle = \langle [x,y]^p\rangle.
\end{equation}
Now, $\langle [x,y^p] \rangle$ is the commutator subgroup of $\langle x,y^p \rangle$, which, by (\ref{eq00}) has order $p^{n-1}$. Therefore, by the inductive assumption, 
\begin{equation}\label{eq01}
[x,y^{p^{n-1}}] = [x,(y^p)^{p^{n-2}}] \ne 1,
\end{equation}
as wanted.

Now, since $C_G(x) \ge \langle A, x\rangle$,  (\ref{eq01}) implies $y^{p^{n-1}}\not\in \langle A, x\rangle$; therefore,
\[ |G| \ge p^n|\langle A, x\rangle| \ge p^{n+1}|A| = p^{2n+1}\]
thus completing the proof. \qed
\end{proof}

This allows to find an exact general bound for odd primes. 

\begin{theorem} For every positive integer $k\ge 1$, let $f_p(k)$ denote the smallest positive integer such that every abelian group of order $p^k$ has an integral of order $p^{f_p(k)}$. Then, if $p$ is an odd prime, $f_p(k) = 2k + 1$.
\label{p:int-p-gp}
\end{theorem}

\begin{proof} Let $k\ge 1$ and let $A$ be an abelian group of order $p^k$. If $d$ is the rank of $A$ and $p^n$  its exponent, then $p^k\ge p^{n+d-1}$. By Lemma \ref{lemma00} there exists an integral $G$ of $A$ such that
\[ |G| = |A|p^{n+d} \le p^{2k+1}.\]
This bound is sharp by Lemma \ref{lemma01}. \qed
\end{proof}

For $p=2$, Lemma \ref{lemma01} fails, as seen by consideration of dihedral
$2$-groups, and we do not have yet the exact value of $f_2(k)$ (see also
\cite{accm}).

\begin{remark} It is clear that
a finite $p$-group (for $p$ a prime number) has a nilpotent integral if and only if it has a finite integral that is a $p$-group. However, not every integrable $p$-group has a nilpotent integral; the quaternion group of order $8$ being the smallest example of an integrable nilpotent group that does not have a nilpotent integral. 
Indeed, by Theorem~\ref{t:not-p-int}, no non-abelian nilpotent group with cyclic centre has a nilpotent integral.
\end{remark}

\subsection{Finitely integrable abelian groups}

This section deals with Problem 10.8 of \cite{accm}. We say that a group $A$ is {\it finitely integrable} if there exists a group $G$ such that $A\cong G'$ and $|G:G'|$ is finite.

 Not every abelian group is finitely integrable: it is proved in \cite{accm} that an infinite direct sum of cyclic $2$-groups with pairwise distinct orders is not finitely integrable.  On the other hand, we have the following simple fact.
 \begin{prop} 
 \label{thm:abelian-finitely-integrable}
 For every abelian group $A$, the direct sum $A\times A$ is finitely integrable.
 \end{prop}

 \begin{proof} Consider the automorphism $\alpha$ of $A\times A$ defined by
 \[(x,y) \mapsto (y, y-x)\]
 for every $(x,y)\in A\times A$. Then the order of $\alpha$ divides $6$. Moreover, for each $(x,y)\in A\times A$, $[(x,y), \alpha] = (y-x, -x)$, hence $[A\times A, \alpha] = A\times A$. Letting $G = (A\times A)\sdir\langle\alpha\rangle$, we have that $G'= A\times A$ has finite index in $G$. \qed
 \end{proof}

\begin{remark}
If $A$ contains no elements of order~$3$, we can use the
automorphism $(x,y)\mapsto(y,-x-y)$, with order $3$, instead.
\label{order3}
\end{remark}

 \begin{cor} Every free abelian group is finitely integrable.
 \end{cor}
 
\subsubsection{Periodic groups} 

In this subsection we consider periodic abelian groups, aiming at a description of the finitely integrable ones. Clearly, if $A$ is a periodic abelian group with no elements of order $2$, then $A$ is finitely integrable via the inversion automorphism; thus, the question reduces to characterizing abelian $2$-groups that are finitely integrable.

Another immediate reduction is to {\it reduced} groups. An abelian $p$-group $A$ is divisible if and only if $A^p = A$, and it is reduced if it contains no non-trivial divisible subgroup. Any abelian $p$-group $A$ has a unique (hence characteristic) maximal divisible subgroup $D$, and $A = D\times B$ with $B$ reduced. 
For $x\in D$, choose $y\in D$ with $y^2=x$; then $[y^{-1},\alpha]=x$, where 
$\alpha$ is the inversion automorphism; so $[D,\alpha]=D$.
It follows that $A$ is finitely integrable if and only if the reduced group 
$B\cong A/D$ is finitely integrable.

So it is enough to consider reduced $2$-groups.

\medskip

We need the following lemma on reduced $p$-groups (only in the case $p=2$).

\begin{lemma}
Let $A$ be a reduced abelian $p$-group, and suppose that $A/A^p$ is finite.
Then $A$ is finite.
\label{l:finite}
\end{lemma}

\begin{proof}
Let $\sigma$ be the $p$th power map on $A$. Then $\sigma$
induces maps (which we also denote $\sigma$) as follows:
\[A/A^p\to A^p/A^{p^2}\to\cdots\to A^{p^m}/A^{p^{m+1}}\to\cdots\]
All these maps are surjective homomorphisms. Since $A/A^p$ is finite, there
exists $m$ such that, for all $n\ge m$, the map
$\sigma:A^{p^n}/A^{p^{n+1}}\to A^{p^{n+1}}/A^{p^{n+2}}$ is an isomorphism.

Suppose first that $A^{p^m}>A^{p^{m+1}}$, and choose an element
$x\in A^{p^m}\setminus A^{p^{m+1}}$. Then successively applying $\sigma$ to $x$
any number $n-m$ of times gives an element which is non-trivial modulo
$A^{p^{n+1}}$, and hence non-trivial; so $x$ has infinite order, a
contradiction.

So we must have
\[A^{p^m}=A^{p^{m+1}}=A^{p^{m+2}}=\cdots.\]
But then the subgroup $A^{p^m}$ is divisible. Since $A$ is assumed reduced,
this implies that $A^{p^m}=1$. But
\[|A|=|A:A^{p^{m}}|\le |A/A^p|\cdot|A^p:A^{p^2}|\cdots|A^{p^{m-1}}:A^{p^m}|
\le|A:A^p|^{m-1},\]
so $A$ is finite, as required. \qed
\end{proof}

If $G$ is an abelian $p$-group, and $n$ a non-negative integer, we set
\[ G[p^n] = \{x\in G\mid x^{p^n} = 1\}\qquad {\rm and} \qquad G^{p^n} = \{ x^{p^n}\mid x\in G\}.\]
These are characteristic subgroups of $G$ and $G/G[p^n] \cong G^{p^n}$. We also write $G^{p^{\omega}} =\bigcap_{n\ge 0}G^{p^n}$. 

\medskip
The simplest reduced groups are direct sums of cyclic $p$-groups. The following is essentially proved in \cite{accm}.

\begin{lemma}\label{lemmaA1} Let $A$ be a homocyclic $2$-group of rank $k \ge 2$: that is
 $A = \Dir_{i\in I}H_i$, 
with $H_i \cong C_{2^n}$ for some fixed $n\ge 1$, and $|I| = k \ge 2$. Then $A$ has an integral of type $A\sdir Q$, with $|Q|$ dividing $21$.
\end{lemma}

\begin{proof} If $k$ is finite, this is done in  \cite{accm}, in the
construction at the start of Section 4 (p.~159) (see also Remark~\ref{order3}
above); indeed, the construction shows that, if $k$ is even, we may take
$|Q|=3$ (see below).

If $k$ is an infinite cardinal, we may find a partition $I = J \cup J'$ with $|J|= |J'|= k$. If $j\mapsto j'$ is a bijection from $J$ to $J'$, then $A = \Dir_{j\in J}(H_j\times H_{j'})$ and there is an automorphism $\alpha$ of order $3$ of $A$,
(again by Remark~\ref{order3}),
fixing every $H_j\times H_{j'}$ and such that $[H_j\times H_{j'}, \alpha] = H_j\times H_{j'}$, so that $A$ is the derived subgroup of $A\sdir\langle\alpha\rangle$. \qed
\end{proof}

\subsubsection{Ulm--Kaplansky invariants}

In this subsection, we suggest another approach to the question of which abelian
$2$-groups are finitely integrable, using the concept of
Ulm--Kaplansky invariants.

Let $p$ be a prime number and $A$ an abelian $p$-group. If $\sigma$ is an ordinal we define $A^{p^{\sigma +1}} = (A^{p^{\sigma}})^p$, while for a limit ordinal $\sigma$, we set $A^{p^{\sigma}} = \bigcap_{\lambda< \sigma} A^{p^{\lambda}}$. Thus, for example
\[ A^{p^{\omega}} = \bigcap_{n\in \N}A^{p^n}.\]
If $A$ is reduced then there exists a smallest ordinal $\tau$, that we call the {\it height} of $A$, such that $A^{p^{\tau}} = 1$.
 
 Let $A$ be a reduced abelian $2$-group of height $\tau$; then for every ordinal $\sigma< \tau$ we define the Ulm section
 \[ U_{\sigma}(A)=\frac{A^{2^{\sigma}}\cap A[2]}{A^{2^{\sigma+ 1}}\cap A[2]}.\]
Clearly, $U_{\sigma}(A)$ is a characteristic section of $A$, and is an
elementary abelian $2$-group. The cardinal number
  $ f_{\sigma}(A) = \rk (U_{\sigma}(A))$
 is called the {\it $\sigma$-th Ulm--Kaplansky invariant} of $A$.

For the next result, we remind the reader of a result about coprime action: if
$H$ is a $p'$-group of automorphisms of a finite abelian $p$-group $G$, then
$G=C_G(H)\times[G,H]$ (see, for example, \cite[Theorem 5.2.3]{Gor}).

 \begin{theorem}\label{propo2} Let $A$ be a reduced $2$-group which is finitely integrable. Then
 \begin{enumerate}
 \item\label{(1)}  only finitely many Ulm--Kaplansky invariants of $A$ are equal to $1$;
 \item\label{(2)} if $A$ has height $\tau > \omega$ then $f_{\sigma}(A)\ne 1$ for every ordinal $\omega \le \sigma < \tau$. 
 \end{enumerate}
 \end{theorem}

 \begin{proof} 
(\ref{(1)}) Let $A$ be a reduced $2$-group of height $\tau$, and let $G$ be an integral of $A$ with $G/A$ finite. Let $G/A = H/A\times Q/A$, where $H$ is a $2$-group and $Q/A$ a finite group of odd order. By coprime action (since $Q/A$ is finite and $A$ an abelian $2$-group), $A = [A,Q] \times C$, where $C = C_A(Q)$. As $[A,Q]\nor G$, we have $(H/[A,Q])' = A/[A,Q]$, hence $C \cong A/[A,Q]$ is finitely $2$-integrable and so, by Corollary \ref{corol-red},
$C/C^2$ is finite. 
Since $C$ is reduced, Lemma~\ref{l:finite} shows that $C$ is finite.
Then, in particular, $C\cap A[2]$ is finite.

Now the action of $Q$ on any section of $A[2]$ is completely reducible; so,
if $C\cap A^\sigma\cap A[2]=C\cap A^{\sigma+1}\cap A[2]$, then $Q$ acts
fixed-point-freely on $(A\sigma\cap A[2])/(A^{\sigma+1}\cap A[2])=U_\sigma(A)$.
Hence only a finite number of sections $U_{\sigma}(A)$ (with $\sigma < \tau$)
may be  centralized by $Q$,
and so only a finite number of such sections are cyclic.

\smallskip
(\ref{(2)}) Let $A$ be a reduced $2$-group of height $\tau > \omega$, and suppose that $f_{\sigma}(A)= 1$ for some ordinal $\omega \le \sigma < \tau$.  Let $Q\le\Aut(A)$ be a finite group of odd order. Then $A = C\times B$, where $C = C_A(Q)$ and $B = [A,Q]$.  Now, $Q$ centralizes the factor $U_{\sigma}(A)$, hence, in particular, $C \cap A^{2^{\omega}} \ne 1$. But, clearly, $A^{2^{\omega}} = C^{2^{\omega}}\times B^{2^{\omega}}$, hence $C^{2^{\omega}}\ne 1$, which implies, in particular, that $C$ is an infinite reduced group. But then, by Corollary \ref{corol-red}, $A/B \cong C$ is not $2$-integrable. This shows that $A$ is not finitely integrable. \qed
 \end{proof}

However, the converse is not true in general (even for countable groups).

\begin{example} 
For every positive integer $n\ge 1$, let $\langle a_n\rangle$ by a cyclic group of order $2^n$, and write $H = \Dir_{n\ge 1}\langle a_n\rangle$. Let 
\[N = \langle a_n^{2^{n-1}}a_{n+1}^{2^n} \mid n\ge 1\rangle,\]
and $H_{\ast} = H/N$. Now, for every $n\ge 1$, 
\[ (a_{n+1}N)^{2^n} = a_{n+1}^{2^n}N = (a_{n}N)^{2^{n-1}} = \ldots = a_1N;\]
 thus
\begin{equation}\label{eqex1}
\langle a_1N\rangle  = \bigcap_{n\ge 1}H_{\ast}^{2^n} =  H_{\ast}^{2^{\omega}}.
\end{equation}
It is not difficult to show that $H_{\ast}/H_{\ast}^{2^{\omega}}\cong H$, and that $f_{\sigma}(H_{\ast}) = 1$ for every finite ordinal $\sigma \le \omega$ (while $f_n(H) =1$ for $n$ a finite ordinal and $f_{\omega}(H)=0$).

We consider $A = H_1\times H_2\times H$, with $H_1\cong H_2\cong H_{\ast}$.  Then $f_n(A) = 3$ for every finite ordinal $n$, and $f_{\omega}(A) = 2$. 

We claim that $A$ is not finitely integrable. Suppose, by contradiction, that there exists a group $G$ with $A= G'$ and $G/A$ finite, and let $Q$ be the odd order component of $G/A$. Now
$Q$ acts on $A$ and we may suppose that the action is faithful. Then $Q$ acts on every section $U_{\sigma}(A)$. As these sections are elementary abelian of order $2^3$ or $2^2$, we have that, for each $\sigma \le \omega$,  $Q/C_Q(U_{\sigma}(A))$ is cyclic of order $1$, $3$ or $7$. Since, by coprime action,
\[ \bigcap_{\sigma\le \omega}C_Q(U_{\sigma}(A)) = C_Q(A[2]) = C_Q(A) = 1,\]
we conclude that $Q$ is the direct sum of cyclic groups of order $3$ or $7$. 

 Suppose that $A^{2^{\omega}}$ is centralized by $Q$; then, writing $Y = C_A(Q)$,
 \[ [A,Q]^{2^{\omega}} \le A^{2^{\omega}} \cap [A,Q] \le [A,Q]\cap Y = 1.\]
 It thus follows from  the decomposition $A = [A,Q]\times Y$ that $A^{2^{\omega}} = Y^{2^{\omega}}$; in particular, $Y$ is infinite and, since it is also reduced (because $A/A^{2^{\omega}} = W/B$ is reduced and $A^{2^{\omega}}$ finite), $Y/Y^2$ is infinite. On the other hand, $Y\cong A/[A,Q]$ is finitely $2$-integrable and so, by Corollary \ref{corol-red},  $Y/Y^2$ is finite, which is absurd. 
 
 Thus, $C_Q(A^{2^{\omega}}) < Q$ and so, as $A^{2^{\omega}}$ is elementary abelian of order $4$, there is an element $x$ of order $3$ in $Q$ such that $[A^{2^{\omega}}, x] = A^{2^{\omega}}$. Let $C = C_A(x)$ and $K = [A,x]$. Then $A^{2^{\omega}} \le K$ and so, arguing as before, $K$ is infinite and, in particular, $K[2]$ is infinite. As $x$ acts fixed point freely on $K$ we deduce that every section $U_{\sigma}(K)$ (with $\sigma\le \omega$) is either trivial or of rank $2$. In particular, there are infinitely many positive integers $n$ such that $f_n (K) = 2$. Now, for every positive integer $n$, 
 \[ 3 = f_n(A) = f_n(C) + f_n(K),\]
 hence there are infinitely many $n$ such that $f_n(C)= 1$. Since $C$ is reduced, it follows from Proposition \ref{propo2} that $C$ is not finitely integrable. However, $C\cong A/K = (G/K)'$, and this is the final contradiction.
\end{example}

\begin{remark} The isomorphism type of a countable reduced abelian $p$-group is determined by its Ulm--Kaplansky invariants (see \cite{FUC}, Theorem 77.3), thus, in principle, the finite integrability of a countable reduced $2$-group should be readable from the sequence of its Ulm--Kaplansky invariants.
\end{remark}

\subsubsection{Torsion-free groups}
The case of abelian torsion-free groups seems much more involved, and we have at the moment little to say.  

\begin{prop} There exist torsion-free abelian groups that are not finitely integrable.
\end{prop}

\begin{proof} For every prime $p$ let $A_p = \Z[\frac{1}{p}]$ (written multiplicatively), and $A = \Dir_{p}A_p$. For every $p$, $A_p$ is the largest $p$-divisible subgroup of $A$ and is therefore characteristic. This implies that every automorphism of finite order of $A$ is an involution. Observe also that $A/A^2$ is infinite.

Now, suppose there exists an integral $G$ of $A$ such that $|G:A|$ is finite and write $C = C_G(A)$. By what was observed above,
as the action of $G/C$ is by automorphisms on $A$ which are involutions,
$G/C$ is a finite $2$-group. Moreover, $Z(C)$ has finite index in $C$ and so $C'$ is finite. As $A$ is torsion-free, we thus have 
$C'\cap A= 1$ and we may well suppose $C'=1$. Let $K = C^2$; we then have $K/A^2 \cong (C/A^2)[2] \ge A/A^2$ (via the homomorphism $cA^2 \mapsto c^2A^2$); in particular,  $C/K$ is infinite and so $AK/K$ is infinite. But then $\overline G = G/K$ is a $2$-group
with $\overline G' = AK/K$ infinite elementary abelian, which (by Corollary \ref{corol-red})
implies that $\overline G/\overline G' \cong G/AK$ is infinite, which is a contradiction. \qed
\end{proof}

\begin{theorem}\label{lemmatf2} Let $G$ be a torsion--free abelian group. If $G$ 
admits an automorphism $\alpha$ of odd order $n$ (possibly trivial) such that
$C_G(\alpha)/C_G(\alpha)^2$ is finite, then $G$ is finitely integrable.
\end{theorem}

\begin{proof}  Let $G$ be an abelian torsion-free group and let $\alpha$ be an
automorphism of odd order $n$ of $G$ such that  $C_G(\alpha)/C_G(\alpha)^2$
is finite. Then, for every $g\in G$,
\[ g^n[g,\alpha]\cdots [g,\alpha^{n-1}]  = g^{1 + \alpha + \ldots + \alpha^{n-1}}\in C_G(\alpha);\]
hence $G^n \le C_G(\alpha)[G,\alpha]$. Then, since $n$ is odd,   
\[ \frac{G}{[G,\alpha]G^2}= \frac{C_G(\alpha)[G,\alpha]G^2}{[G,\alpha]G^2} \cong \frac{C_G(\alpha)}{C_G(\alpha)\cap [G,\alpha]G^2}\]
 is finite by hypothesis. Denote by $\lambda$ the inversion automorphism of $G$; then $\beta = \alpha\lambda$ is a fixed-point-free automorphism of order $2n$ of $G$, whence  $[G, \beta] \cong G/C_G(\beta) = G$. But 
\[[G, \beta] \ge [G, \alpha][G,\lambda] = [G, \alpha] G^2\]
 has finite index in $G$. Thus, by letting $H = G\sdir\langle \beta\rangle$, we have $H' = [G,\beta]\cong G$, showing that $G$ is finitely integrable. \qed
\end{proof}

\begin{cor}\label{cortf1} Every torsion--free abelian group of finite rank is finitely integrable.
\end{cor}

Proposition \ref{lemmatf2} implies, in particular, that a torsion-free abelian group admitting a fixed-point-free automorphism of odd order is finitely integrable. However, because of the abundance of indecomposable torsion--free abelian groups, nothing similar to Theorem \ref{propo2} is to be expected in this case. For instance, examples of indecomposable (as direct sums) torsion-free groups admitting a fixed-point-free automorphism of order $3$ may be found in Chapter XVI of \cite{FUC}: e.g.$\!$ Example 2 at page 272. That group is indeed of rank $2$ and so it is finitely integrable anyway by Corollary \ref{cortf1}; it is however possible to extend that construction to obtain indecomposable groups of infinite rank with a fixed-point-free automorphism of order $3$. 

\begin{example}
Let $\mathcal P$ be a partition into an infinite number of infinite disjoint subsets of the set of all primes $q\equiv 1\!\!\pmod 6$. For every $I\in \mathcal P$, let $A_I$ be a torsion--free group as constructed, with respect to the set of primes in $I$, in Example 2 of page 272 in \cite{FUC}. The groups $A_I$ (for $I\in \mathcal P$) are indecomposable, pairwise non-isomorphic, and admit a fixed point free automorphism of order $3$.  The direct sum $G = \Dir_{I\in \mathcal P}A_I$  of these groups admits a fixed-point-free automorphism of order $3$, hence it is finitely integrable, but it is not decomposable as the direct sum of two (or more) isomorphic subgroups, and is such that $G/G^2$ is infinite.
\end{example}

\subsection{Finitely generated integrals}

We know (see \cite{accm}) that a finitely generated abelian group has a finitely generated integral (and even a nilpotent one, by Theorem~\ref{t:nilpotent-integral}). 
On the other hand it is well known that the derived subgroup of a finitely generated group need not be finitely generated: (for instance, consider the derived
subgroup of the non-abelian free group of rank 2), 
and thus it makes sense to ask which abelian groups have an integral which is finitely generated; a question that goes back to P. Hall \cite{HALL}. In his paper, Hall found necessary conditions for an abelian group to occur as a normal subgroup with polycyclic factor of a finitely generated group. Much later, in \cite[Theorem 1.1]{MIKOL}, Mikaelian and Olshanski proved that the class of abelian groups described by Hall is precisely that of groups which  are isomorphic to a subgroup of the derived group of a finitely generated (in fact, $2$-generated) metabelian group. In the same paper (Theorem 1.3 and Example 5.1), Mikaelian and Olshanskii show that not all such groups may be embedded as the derived group of a finitely generated group. 

For every set $\pi$ of primes, denote by $D_{\pi}$ the set of all rational numbers whose denominator is a positive $\pi$-number: also, set $D_{\emptyset}=\Z$. From \cite{MIKOL} it is not difficult to retrieve the following necessary condition.

\begin{prop} Let $G = T\times D$ where
\begin{itemize}
\item  $T$ is a finite or countable abelian group of finite exponent, and 
\item $D$ is the direct sum of a family of groups $\{ D_{\pi_i}\mid i\in I\}$, with $I$ finite or countable and $\bigcup_{i\in I}\pi_i$ finite. 
\end{itemize}
Then there exists a finitely generated group $H$ such that $G\cong H'$.
\end{prop}

We have no idea whether this condition is also necessary; thus, a full characterization of abelian groups that have a finitely generated integral seems to be still open.


\section{Varieties of groups}
\label{sec:varieties}

\subsection{The integral of a variety}

The starting point of this subsection was Problem 10.13 of \cite{accm}. We know that, if $\mathbf{V}$ is a variety of groups, then the class of all
integrals of groups in $\mathbf{V}$ is a variety \cite{accm}.
In fact we can say a bit more. 

\begin{prop}\label{varieties}
The class of integrals of groups in the variety $\mathbf{V}$ is a variety;
indeed it is the product variety $\mathbf{VA}$, where $\mathbf{A}$ is the
variety of abelian groups.
\end{prop}

\begin{proof}
The product variety $\mathbf{VA}$ consists of all groups which have a normal
subgroup in $\mathbf{V}$ with quotient in $\mathbf{A}$; it is a variety
(Neumann~\cite[21.11]{hn}). Clearly, if $H$ is an integral of a group
$G\in\mathbf{V}$ then $H\in\mathbf{VA}$.

Conversely, suppose that $H\in\mathbf{VA}$, then there is a subgroup
$N\unlhd H$ with $N\in\mathbf{V}$ and $H/N$ abelian; so $H'\le N$ and so
$H'\in\mathbf{V}$, since $\mathbf{V}$ is subgroup-closed. \qed
\end{proof}

We call $\mathbf{VA}$ the \emph{integral} of $\mathbf{V}$.

\medskip

Let $G$ be a finite group, $\mathbf{V}$ the variety generated by $G$. Then
$\mathbf{V}$ is finitely based, by the Oates--Powell Theorem, see
\cite[52.11]{hn}. Is the integral $\mathbf{W}$ of $\mathbf{V}$ also
finitely based?

A basis for $\mathbf{W}$ consists of the identities
\[v(x_1,\ldots,x_m)=1,\]
where $v$ is an identity of $\mathbf{V}$ and $x_1,\ldots,x_m$ are elements
of the relevant free group which are products of commutators \cite[21.12]{hn}.
This set is infinite. So, to get a finite basis
for the identities, we would require that there is a positive integer $k_0$
with the property that, if each identity $v(u_1,\ldots,u_r)$ holds when each
$u_i$ is a product of at most $k_0$ commutators, then it holds in general
without this restriction.

This leads us to the following definition.
Let $S$ be a symmetric subset of a group $G$ (that is, $S=S^{-1}$). Let
\[B_k(S)=\bigcup_{i=0}^kS^i\]
be the ball of radius $k$ in the Cayley graph of $G$ with respect
to~$S$.

We say that a group identity $w=1$ has \emph{gauge} $k$ if, whenever
$S$ is a symmetric generating set for a finite group $G$ and the identity
$w=1$ holds in $B_k(S)$, then the identity holds in $G$.

The \emph{gauge} of a variety $\mathbf{V}$ is the smallest $k$ such that, if
the identities $w=1$ defining the variety all hold in $B_k(S)$, where $S$
is a symmetric generating set for a group $G$, then $G\in\mathbf{V}$. (This is
in general weaker than requiring that all the identities defining $\mathbf{V}$
have gauge at most $k$. But if $\mathbf{V}$ is finitely based, then it is
defined by a single identity, and we can require this identity to have
finite gauge.)

\begin{remark} 
Why symmetric? First, it makes a difference. Consider
the symmetric group $S_n$ with the usual generating set $a=(1,2)$ and
$b=(1,2,\ldots,n)$. Consider the metabelian
identity $[[x,y],[z,w]]=1$. If we use the generating set
$S=\{a=a^{-1},b,b^{-1}\}$, and
we substitute $x=a$, $y=b$, $z=a$, $w=b^{-1}$, the identity does not hold.
But if we use the generating set $\{a,b\}$,
any substitution of generators
satisfies one of $x=y$; $z=w$; $\{x,y\}=\{z,w\}$. In each case the identity
is satisfied.

Second, in our application, the generating sets that arise will be symmetric.
For the
derived group of a group $G$ is generated by all commutators $[x,y]$
for $x,y\in G$, and $[x,y]^{-1}=[y,x]$.
\end{remark}

Our earlier considerations give the following result.

\begin{theorem}
Let $\mathbf{V}$ be a variety which is finitely based and has
finite gauge. Then the integral of $\mathbf{V}$ is finitely based.
\end{theorem}

\begin{prop}
Each of the following varieties has gauge~$1$: the variety $\mathbf{A}$ of
abelian groups, the variety $\mathbf{A}_m$ of abelian groups of exponent
dividing $m$, and the variety $\mathbf{N}_c$ of nilpotent groups of class at
most $c$.
\end{prop}

\begin{proof}
If the generators of $G$ commute, then $G$ is abelian; if in addition the
generators have order dividing $m$, then $G$ has exponent dividing $m$.

The variety $\mathbf{N}_c$ is defined by the identity
$[x_1,x_2,\ldots,x_{c+1}]=1$, where the commutator is left-normed,
defined inductively by
\[[x_1,\ldots,x_{k+1}]=[[x_1,\ldots,x_k],x_{k+1}]\]
for $k\ge2$.  The proof is by induction on $c$, the first part of the
proposition giving the case $c=1$.

So let $G$ be a group with symmetric generating set $S$ satisfying the
identity $[x_1,x_2,\ldots,x_{c+1}]=1$. This identity
shows that the element $[x_1,\ldots,x_c]$, for $x_1,\ldots,x_c\in S$, commutes
with every generator, and so belongs to $Z(G)$. This shows that $G/Z(G)$
with generating set $\overline{S}=SZ(G)/Z(G)$ has the property that all
generators satisfy $[x_1,\ldots,x_c]=1$; by induction, $G/Z(G)$ is nilpotent
of class at most $c-1$, so $G$ is nilpotent of class at most $c$. \qed
\end{proof}

\begin{prop}
The identity $x^2=1$ has gauge $2$ (and not $1$).
\end{prop}

\begin{proof}
If $x^2=y^2=1$, then $(xy)^2=1$ if and only if $x$ and $y$ commute. So, if
all generators and their pairwise products have order~$2$, then all pairs of
generators commute, and $G$ is abelian of exponent $2$. 
But of course there
are non-abelian groups generated by elements of order $2$.
Moreover, it is straightforward to verify that if $g^2=1$ for every $g \in B_2(S)$, then $g^2=1$ for every $g \in G$.
\qed
\end{proof}

In the other direction, we have the following:

\begin{prop}
The variety of metabelian groups has infinite gauge.
\end{prop}

\begin{proof}
Let $R = \langle a, b\rangle$ be a non-abelian simple group, and consider the restricted wreath product $G = R\wr\langle x\rangle$, where $x$ is an element of infinite order. Denote by $E$ the base of the wreath product (i.e. the direct sum
of all coordinate subgroups $R^{x^z}$ for $z\in \mathbb{Z}$); thus
$E = G'$ and $G$ is the semidirect product $E\rtimes\langle x\rangle$.

Let $n$ be a positive integer, and $N = 4n+1$.  Let $c = b^{x^N}$; then \[S = \{a, a^{-1}, c, c^{-1}, x, x^{-1}\}\] is a symmetric set of generators of $G$. As introduced earlier, for $k\ge 1$ denote by $B_k = \bigcup_{i= 0}^{k}S^i$ the ball of radius $k$ and, for $g\in G$, by $\ell (g)$ the length of $g$ as a word in $S$, 
that is, its distance from the identity in the Cayley graph of $G$ with generating set $S$.

For each $0\le t$ let
\[A_t = \langle a^{(x^z)}\mid |z|\le t\rangle,\quad C_t = \langle c^{(x^z)}\mid  |z|\le t\rangle \ \ \mbox{and} \ \ W_t = \langle A_t, C_t\rangle.\]
Observe that $W_t^x\cup W_t^{x^{-1}}\subseteq W_{t+1}$.

\paragraph{Claim:} $B_t \cap E \subseteq W_{\lfloor t/2\rfloor}$. We prove this by induction on $t$, the fact being clear for $t=1$. Thus, let $t\ge 2$ and suppose $u = vy\in B_t\cap E$ with $v\in B_{t-1}$ and $y\in S$. If $y\in\{ a, a^{-1}, c, c^{-1}\}$, then $v\in B_{t-1}\cap E$ and we are done by induction. Let $u = vx$; then since $u\in E$ there is, in the writing of $v$ as a word of length $t-1$ in $S$, an occurrence of $x^{-1}$ somewhere; that is
\[u  = v_1x^{-1}v_2x,\]
with $v_1, v_2\in E$ and $1\le \ell (v_2) \le t-2$.  By inductive assumption, $v_2\in W_{\lfloor t/2\rfloor-1}$ and so $x^{-1}v_2x\in W_{\lfloor t/2\rfloor}$; since  $v_1 \in E\cap B_{t-3}$, we have $u \in W_{\lfloor t/2\rfloor}$, as wanted.

\medskip
Now, observe that $A_t, C_t$ are abelian for every $t$; moreover,when $t \le 2n$,  $A_t$ and $C_t$ have trivial intersection and commute element-wise, so that, for $t\le 2n$,
\[ W_t = \langle A_t, C_t\rangle = A_t\times C_t\]
is abelian.
Let finally $g_1,g_2, g_3, g_4\in B_n$; then $[g_1,g_2], [g_3,g_4]\in B_{4n}\cap E \subseteq W_{2n}$, and so
\[[[g_1,g_2], [g_3,g_4]] = 1.\]
Thus the metabelian identity holds in $B_n$ but not in the whole group, 
meaning that its gauge is greater than $n$. This holds for all $n$, so the
gauge is infinite. \qed
\end{proof}

However, a metabelian variety generated by a finite group may have finite
gauge. For example, consider the variety $\mathbf{V}$ generated by the
group $S_3$. It is known that $\mathbf{V}=\mathbf{A}_3\mathbf{A}_2$, and
as a basis for the identities we may take
\[x^6=[x^2,y^2]=[x,y]^3=[x^2,[y,z]]=[[x,y],[z,w]]=1.\]
We claim that this variety has gauge~$1$.
For if the generators of a group satisfy these identities, then their
squares and commutators commute and have order dividing~$3$, so generate a
group in $\mathbf{A}_3$; the quotient is in $\mathbf{A}_2$.

\bigskip

\subsection{Varieties with every group integrable}
\label{sec:integrable-varieties}

A possibly easier question \cite[Problem 10.15]{accm}, on which we have some
results, is the following.

\begin{question}
Is there a variety of groups, other than a variety of abelian groups, with
the property that every group in the variety is integrable?
\end{question}

The class $\mathbf{B}_p\cap\mathbf{N}_2$ is a candidate for prime $p$,
where $\mathbf{B}_p$ is the variety of groups of exponent $p$, and
$\mathbf{N}_2$ the variety of nilpotent groups of class at most $2$.
We prove that it does indeed have the required property.

\begin{theorem} Let $p$ be an odd prime. Then every group in the variety of groups of exponent $p$ and nilpotency class at most $2$ has an integral.
\end{theorem}

\begin{proof} Let $G$ be a group of exponent $p$, for $p$ an odd prime, and nilpotency class at most $2$. The set $G$ becomes a Lie $GF(p)$-algebra $L_G$ by setting, for $x,y\in G$,
\[ x + y = xy[x,y]^{\frac{1}{2}}\]
and letting the group commutator be the Lie product (this is essentially the simplest case of Malcev's correspondence)
where by $y^{\frac{1}{2}}$ we mean the preimage of the isomorphism $a \mapsto a^2$.

If $\{x_1, \ldots x_n\}$ is a minimal set of generators of the group $G$, then 
\[L_G = L_1 \oplus L_2\]
where $L_1$ is the $GF(p)$--space spanned by $x_1, \ldots, x_n$ and $L_2 = [L_G, L_G]$. The map
\[x\oplus z \mapsto (-x) \oplus z\]
(for $x\in L_1$, $z\in L_2$) is an automorphism of $L_G$, to which it corresponds an automorphism $\alpha $ of the group $G$ of order $2$. Thus, $\alpha$ induces the inversion map on $G/G'$, hence, letting $H = G\sdir\langle \alpha\rangle$, we have $H' =G$. \qed
\end{proof}

A different perspective, suggested by the proof of the result about the orders
for which every group is integrable \cite[Theorem 7.1]{accm}, is to ask whether
every group in the variety $\mathbf{A}_p\mathbf{A}_q$ is integrable, where 
$p$ and $q$ are primes with $q\nmid p-1$. On this, we can prove the following:

\begin{theorem}\label{le3} Let $p,q$ be distinct primes such that $p\nmid q-1$. 
Then every finite group in $\mathbf{A}_q\mathbf{A}_p$ has an integral.
\end{theorem}

We introduce the principal argument for the proof in a separate Lemma.

\begin{lemma}\label{le4} Let $p,q$ be distinct primes such that $p\nmid q-1$ and $m = ord_p(q)$. 
Let $G$ be a
finite group in $\mathbf{A}_q\mathbf{A}_p$, and let $Q$ be the largest normal $q$-subgroup of $G$ and suppose $Q\ne G$;  then there exists an automorphism $\alpha$ of $G$ of order $m$ which acts as a non-trivial power on $G/Q$.
\end{lemma}

\begin{proof} We proceed by induction on $|G|$. Since $m\mid p-1$, the claim is obvious if $Q$ is trivial since the map $\alpha(g)= g^{\frac{1-p}{m}}$ fits the requirements.
Let $P$ be a Sylow $p$-subgroup of $G$. By assumption, $P$ is not trivial. 

Suppose $C = C_P(Q)\ne 1$. Then $P = C\times P_1$, where $P_1=[P,Q]$, and, by the inductive assumption,  $G_1= QP_1$ admits an automorphism $\alpha$ acting as a power of order $m$ on $G_1/Q$. Now, $G = C\times G_1$ and by letting $\alpha$  act on $C$ by the same power it acts on $G_1/Q$ we are done.
Thus, we now suppose $C_P(Q) = 1$.  

If $Q$ is indecomposable as $GF(q)P$-module, then $P$ is cyclic, $|Q| = q^m$, and  $G = QP$ may be represented as a group of affine transformations of the field $GF(q^m)$. 
(The order formula for the general linear group shows that the group of
affine transformations of $GF(q^m)$ contains a $p$-group whose order is the
$p$-part of $\mathrm{GL}(m,q)$. Thus $G$ is conjugate to a subgroup of this
affine group.) Then a Galois automorphism of order $m$ induces an automorphism
of $G$ that acts as a non-trivial power on $G/Q\cong P$.

Now, suppose $Q = Q_1\times Q_2$, with $Q_1, Q_2$ non-trivial normal subgroups of $G$. For $i = 1,2$, let $G_i = G/Q_i$. By inductive assumption, each $G_i$ admits an automorphism $\sigma_i$ of order $m$ acting as a power on $G_i$ modulo $Q/Q_i$. By possibly replacing $\sigma_2$ with one of its powers, we have that $\sigma_1, \sigma_2$ induce the same power on the appropriate quotients. Then $\sigma = (\sigma_1, \sigma_2)\in\Aut(G_1\times G_2)$ acts as a power automorphism on $G_1\times G_2$ modulo its largest (normal) $q$-subgroup $Q_0$. Now, we have a natural injective homomorphism $\pi : G\to G_1\times G_2$, and $\pi (G) > \pi (Q) = Q_0$. Since $\sigma$ acts as a power on $(G_1\times G_2)/Q_0$, it in particular fixes $\pi (G)$. Hence $\sigma$ induces an automorphism of $G$ of order $m$ that acts as a power on $G/Q$. \qed
\end{proof}

\begin{proof}\textbf{of Theorem \ref{le3}}
Let $G = QP$ be a finite group in $\mathbf{A}_q\mathbf{A}_p$, where $Q$ is a normal (elementary abelian) $q$-subgroup and $P$ a Sylow $p$-subgroup of $G$. 
As, by coprime action, $Q = C_Q(P)\times [Q,P]$, we have $G = C_Q(P) \times [Q,P]P$. Hence we may well assume $Q = [Q,P]$,
since $C_Q(P)$ is an abelian direct factor and so it is integrable.

Let $m = ord_p(q)$. By Lemma \ref{le4}, $G$ admits an automorphism $\alpha$ of order $m$ acting as a non-trivial power on $G/Q$. By a standard fact, 
we may assume that $\alpha(P)=P$, so that $[P,\alpha] = P$. Let $H = G\sdir\langle \alpha\rangle$, so that clearly $H' \le G$.  Then
\[ H' \ge [Q,P][P,\alpha]  = QP = G.\]
and we are done. \qed
\end{proof}


\section{Self-integrating classes of groups}
\label{sec:self-integrating}

We now consider integrals within certain classes of groups. This section is related to Problem 10.15 of \cite{accm}. 

Let $\mathcal{C}$ be a class of groups. (We use this phrase to mean that
$\mathcal{C}$ is isomorphism-closed.) The strongest property we might
require is to ask the following: if a group $G\in\mathcal{C}$ is integrable,
then every integral of $G$ is in $\mathcal{C}$. It is reasonable to require
that $\mathcal{C}$ is subgroup-closed, otherwise there will be uninteresting
examples such as the class of non-abelian groups. In this case, $\mathcal{C}$
contains the trivial group, and hence all abelian groups, and hence (by
induction) all soluble groups. But certainly the class of all soluble groups
has our properties. Again we then get uninteresting examples such as groups
which have at most one nonabelian composition factor, this factor being
$A_5$.

A more sensible definition is the following. We say that a class $\mathcal{C}$
of groups is \emph{self-integrating} if, whenever $G$ is an integrable group in
$\mathcal{C}$, then $G$ has an integral in $\mathcal{C}$.

We saw in~\cite{accm} that the class of finite groups, and the class of
finitely generated groups, are both self-integrating.

Obviously, the class of soluble groups is self-integrable, as well as every class which is closed by extensions by abelian groups (e.g. the class of amenable groups). 
However, not everything is so trivial.

\begin{lemma} The following classes of groups are self-integrating.
\begin{enumerate}
\item\label{11} finite groups;
\item\label{12} finitely generated groups;
\item\label{13} polycyclic groups, and more generally groups satisfying Max;
\item\label{14} finitely generated, residually finite groups.
\end{enumerate}
\end{lemma}

\begin{proof}
(\ref{11}) and (\ref{12}) are true by \cite[Theorem 2.2, Proposition 9.1]{accm}.
As (\ref{13}) follows easily from (\ref{12}), we only consider (\ref{14}).

Thus, let $G$ be a finitely generated residually finite group.  For every $n\ge 1$, let  $K_n$ be the intersection of all subgroups of $G$ of index at most $n$; then, each $K_n$ is characteristic. Moreover, since $G$ is finitely generated, $G/K_n$ is finite for every $n$, and, $\bigcap_{n\ge 1}K_n = 1$, because $G$ is residually finite. 

Now, suppose that  $G$ has an integral $H$, which, by point (\ref{12}), we may suppose is finitely generated. For every $n\ge 1$, the commutator subgroup of $H/K_n$ is finite; hence, since $H/K_n$ is finitely generated, $Z_n/K_n = Z(H/K_n)$ has finite index in $H/K_n$ (see for instance exercise 14.5.7 in Robinson's book). Also, $Z_n/K_n$ is a finitely generated abelian group, and so there exists a subgroup $C_n/K_n$ of $Z_n/K_n$ with $[Z_n:C_n]$ finite and $C_n\cap G=K_n$. 
For take $C_n/K_n$ to be a complement for the torsion
subgroup of $Z_n/K_n$, noting that $G/K_n$ is contained in this torsion
subgroup since it is finite.
Observe that $C_n\nor H$ and that $[H:C_n]$ is finite. Setting $C = \bigcap_{n\ge 1}C_n$, we have $C\nor H$ and 
\[C\cap G = \bigcap_{n\ge 1}(C_n\cap G) = \bigcap_{n\ge 1}K_n = 1.\]
Thus, $(H/C)' = GC/C \cong G$, and we are done since $H/C$ is residually finite (and finitely generated). \qed
\end{proof}

\begin{question}
Which other ``natural'' classes of groups are self-integrating?
For instance: periodic groups, torsion--free groups, linear groups, residually finite groups in general, 
virtually free groups, \dots
\end{question}

We show in this paper that the class of finite $p$-groups, and the class of
residually finite groups, are both not self-integrating
(Theorem~\ref{t:not-p-int} for $p$-groups, Lemma~\ref{lemcar1} and
Proposition~\ref{procar2} below for residually finite groups).


\section{Profinite groups and Cartesian products}
\label{sec:profinite}

\subsection{Profinite and abstract integrals}
\label{ss:profabs}

 Let $G$ be a compact topological group.
 Recall that $G$ is said to be profinite if the following equivalent conditions are satisfied (cf. \cite[Lemma~2.1.1 and Theorem~2.1.3]{RibZal:book}):
  \begin{itemize}
  \item[(i)] there exists an inverse system $\{G_i, \varphi_{ij}\colon G_i\to G_j\:\mid\:i,j\in I,i>j\}$ of finite groups such that $G=\varprojlim_{i\in I}G_i$, and $\{\ker(\varphi_i)\mid i\in I\}$ is a basis of open neighbourhoods of the identity, where $\varphi_i\colon G\to G_i$ denotes the canonical epimorphism for every $i\in I$;
  \item[(ii)] there exists a basis of open neighbourhoods of the identity $\mathcal{U}$ consisting of normal subgroups of $G$, such that $G=\varprojlim_{N\in\mathcal{U}}G/N$.
 \end{itemize}
 (For the definition of inverse system and projective limit, see \cite[\S~1.1]{RibZal:book}.)
Observe that a subgroup of a compact topological group is open if, and only if, it is closed and of finite index (cf., e.g., \cite[Lemma~2.1.2]{RibZal:book}).
For a profinite group $G$ and a subset $X\subseteq G$, $\overline X$ will denote the closure of $X$.

There are two notions of derived group in the class of
profinite groups: either the abstract (the subgroup generated by commutators)
or the topological (the closure of the preceding). 
We say that a profinite group $G$ has a profinite integral $K$ if $K$ is a profinite group and $G$ is the topological derived subgroup $\overline{K'}$.

We show that a profinite group which has finite index in some integral has a
profinite integral, and that a finitely generated profinite group which has an
integral has a profinite integral. However, in general it is not true that an
integrable profinite group has a profinite integral (see Theorem \ref{teoprof1},
Lemma \ref{lemcar1} and Proposition \ref{procar2}).

We begin with a known remark that we will use throughout this section.

\begin{remark}
\label{rem:nikolov-segal}
Let $G$ be a topologically finitely generated profinite group.
(Henceforth, ``finitely generated profinite group'' will be intended in the topological sense.)
By a remarkable result of Nikolov and Segal (cf. \cite{ns1}), $G$ boasts the following properties:
\begin{itemize}
 \item[(i)] the abstract derived subgroup of $G$ is closed, i.e., $\overline{G'}=G'$;
 \item[(ii)] all the subgroups of finite index of $G$ are open, and there are only finitely many of them of a given index.
\end{itemize}
For every positive integer $n$, let $G(n)$ denote the intersection of all the subgroups of $G$ of index at most $n$. 
From property~(ii), one deduces that $G(n)$ is a closed characteristic subgroup of finite index --- and thus $G(n)$ is also open. 
It is straightforward to see that for every ascending sequence of positive integers $i_1<i_2<\ldots<i_n<\ldots$, 
the family $\mathcal{U}=\{G(i_n)\mid n\geq1\}$ is a basis of open neighbourhoods of the identity consisting of normal subgroups of $G$, and thus $$G=\varprojlim_{n\geq1}G/G(i_n).$$
We use this notation
throughout the present section.
\end{remark}

\begin{prop}\label{prop:profiniteintegral}
 Let $(I,\leq)$ be a directed set, and let $\{G_i,\varphi_{ij}\mid i,j\in I\}$ be an inverse system of finite groups with associated profinite group $G=\varprojlim_{i}G_i$.
 Suppose that there exists an inverse system of finite groups $\{K_i,\psi_{ij}\mid i,j\in I\}$, with associated profinite group $K=\varprojlim_{i}K_i$, such that one has an isomorphism $\tau_i\colon G_i\to K_i'$ for every $i\in I$.
 Then $G\simeq \overline{K'}$.
\end{prop}

\begin{proof}
 The short exact sequences of finite groups 
 \[\xymatrix{ 1 \ar[r] & G_i\ar[r]^-{\tau_i} & K_i\ar[r] &  K_i/G_i\ar[r] &1 },\]
 for every $i\in I$, yield a monomorphism of profinite groups $\tau\colon G\to K$ such that $\tau(G)$ is a closed normal subgroup of $K$, and $\psi_i\circ\tau=\tau_i\circ\varphi_i$ for every $i\in I$ (cf. \cite[Proposition~2.2.4]{RibZal:book}). Moreover, $K/\tau(G)$ is an abelian profinite group, as every quotient $K_i/G_i$ is abelian.
 Therefore, one has the inclusion $\tau (G)\supseteq\overline{ K'}$.
 
 On the other hand, for every $i\in I$ let $N_i$ be the kernel of the canonical epimorphism $\psi_i\colon K\twoheadrightarrow K_i$.
 Then $\{N_i\:\mid\: i\in I\}$ is a basis of open neighbourhoods of the identity.
 Now pick an arbitrary element $x$ of $\tau(G)$, and an arbitrary open neighbourhood $U\subseteq K$ of $x$.
 Thus, there exists $j\in I$ such that the coset $xN_j$ --- which is an open neighbourhood of $x$ --- is contained in $U$.
Since the diagram
\[\xymatrix{ G\ar[d]_{\varphi_j}\ar[rr]^-{\tau} && K\ar[d]^{\psi_j} \\
  G_j\ar[r]^-{\tau_j} & K_j'\ar@{^{(}->}[r] & K_j }\]
commutes, one has $(\psi_j\circ\tau)(G)\subseteq K_j'$.
Up to rewriting the images of the homomorphism $\psi_j$ as cosets of $N_j$, it makes sense to write $xN_j\in K_j'$.
Since 
 \[
  K_i'=(K/N_i)'=(K'\cdot N_i)/N_i
 \]
for every $i\in I$, one has that $xN_j=hN_j$ for some $h\in K'$.
Therefore $h\in xN_j\subseteq U$ --- in other words, every element of $\tau(G)$ is arbitrarily close to $K'$.
Hence, $\tau(G)\subseteq \overline{K'}$. \qed
\end{proof}

\begin{theorem}
\label{thm:finitely-integrable-profinite}
Let $G$ be a profinite group and $H$ an integral of $G$ (as group) with
$|H:G|$ finite. Then there exists a profinite group $K$ which is both a profinite integral and an abstract integral of $G$, i.e., $G=\overline{K'}=K'$.
\end{theorem}

\begin{proof}
 Pick a basis of open neighbourhoods of the identity $\mathcal{U}=\{N_i\mid i\in I\}$, with $(I,\leq)$ a directed set such that $N_i\leq N_j$ for every $i,j\in I$ such that $i\geq j$, consisting of normal subgroups of $G$, and set $G_i=G/N_i$. 
 Then $G_i$ is a finite group for every $i\in I$, and $G=\varprojlim_{i}G_i$.
 For every $i\in N$, set
 \[
  K_i:=\bigcap_{hG\in H/G}h^{-1}N_ih.
 \]
Since $H/G$ is finite by hypothesis, $K_i$ is the intersection of a finite number of open subgroups of $G$, and thus it is open.
Moreover, $K_i\leq K_j$ for every $i,j\in I$ such that $i\geq j$, and for every open subset $U$ of $G$ there exists $i\in I$ such that $U\supseteq N_i\supseteq K_i$.
Hence, $\mathcal{K}=\{K_i\mid i\in I\}$ is a basis of open neighbourhoods of the identity in $G$, and $G=\varprojlim_{i} G/K_i$.

 For every $i\in N$, one has that $[H:K_i]=[H:G][G:K_i]$ is finite.
 Thus, $\{H/K_i, H/K_i\twoheadrightarrow H/K_j\text{ for }K_i\leq K_j\}$ is an inverse system of finite groups, and we may define the profinite group $K=\varprojlim_{i}H/K_i$, which has $\mathcal{K}$ as a basis of open neighbourhoods of the identity consisting of normal subgroups.
 Since $\bigcap_{i\in I}K_i=\{1\}$, the definition of $K$ yields a monomorphism of groups $\phi\colon H\hookrightarrow K$, and hence $H$ may be considered as a subgroup of $K$, so that $$G=H'\leq K'.$$
Observe that $G$ is an open (and thus also closed) subgroup of $K$, as 
\[
 G=\bigcup_{gK_i\in G/K_i}gK_i\qquad\text{for each }i\in I
\]
and every $gK_i\in G/K_i$ is an open subset of $K$. 
 
 On the other hand, for every $i\in I$ let $\varphi_i\colon K\to H/K_i$ denote the canonical epimorphism. 
 Then $H/K_i\simeq K/\ker(\varphi_i)=K/K_i$, while $G/K_i=(H/K_i)'$ by hypothesis.
 Altogether,
 \[
  \dfrac{K}{G}\simeq \dfrac{K/K_i}{G/K_i}\simeq\dfrac{H/K_i}{(H/K_i)'} \qquad\text{ for each }i\in I,
 \]
hence $K/G$ is abelian. 
Consequently, $G$ contains $K'$, and thus also $\overline{K'}$, as $G$ is a closed subgroup of $K$.
Therefore, $G=K'=\overline{K'}$. \qed
\end{proof}

\begin{theorem}
\label{thm:profinite-fg-integral}
Let $G$ be a finitely generated profinite group which is integrable as
abstract group. Then there exists a finitely generated profinite group $K$ which is both a profinite integral and and abstract integral of $G$, i.e., $G=\overline{K'}=K'$.
\end{theorem}

\begin{proof}
Let $H$ be an integral of $G$.
Take a set of generators $g_1,\ldots,g_s$ of $G$ as profinite group.
By the proof of \cite[Proposition 9.1]{accm}, we can find a finitely generated
abstract subgroup $T$ of
$H$ so that $\langle g_1,\ldots,g_s \rangle = T'$. 
We define
$H^*:=GT$ and $L:=(H^*)' \le G$. 

Let $G(n)$ be defined as in Remark \ref{rem:nikolov-segal}. For any $n \in \mathbb{N}$ we observe that $G(n) \le G(n)L \le G$, so $G(n) L$ 
is a finite index subgroup of $G$.
By Remark~\ref{rem:nikolov-segal}, $G(n) L$ is an open --- and thus also closed --- subgroup of $G$.
Moreover, since $G(n)L \ge \langle g_1,\ldots,g_s\rangle$, we have
$$G\ge G(n)L=\overline{G(n)L} \ge \overline{\langle g_1,\ldots,g_s\rangle}=G$$ and so $G=G(n)L$ for every $n \in \mathbb{N}$.

Moreover, $H^*/G =GT/G \cong T/(T \cap G)$ is finitely generated as an abstract group.

Notice that $$(H^*/G(n))'=L\cdot G(n)/G(n)=G/G(n)$$ is a finite group and so, by the same argument of the proof of \cite[Theorem 2.2]{accm} we have that $Z(H^*/G(n))=Z_n/G(n)$ for a suitable subgroup $Z_n \le H^*$ so that $H^*/Z_n$ is a finite group.
Since $H^*/G$ is finitely generated as an abstract group and $G/G(n)$ is finite, we have that $H^*/G(n)$ is finitely generated as an abstract group.
Moreover, since $Z_n/G(n)$ has finite index in $H^*/G(n)$, then $Z_n/G(n)$ is an abelian finitely generated
abstract group.

Let $N_2 \le Z_2$ be such that $N_2 \cap G = G(2)$ and $[Z_2:N_2]<\infty$ (for example, take $N_2/G(2)$
to be a complement to the torsion subgroup of $Z_2/G(2)$).
Now assume we have constructed
$N_n \le N_{n-1} \le \ldots \le N_2$ so that $N_n \cap G = G(n)$ and $[H^*:N_n]<\infty$. 
Since $(Z_{n+1} \cap N_n)/G(n+1)$ is central
 and has finite index in $H^*/G(n+1)$,
we can find $N_{n+1} \le N_n$ so that
\[
\frac{N_{n+1}}{G(n+1)} \le \frac{Z_{n+1}\cap N_n}{G(n+1)} \; \; \mathrm{and} \; \; N_{n+1} \cap G = G(n+1) \; \; \mathrm{and} \; \; [H^*:N_{n+1}]<\infty.
\]
Observe that $N_n$ is normal in $H^*$ for all $n\geq2$, as $G(n)\leq N_n \leq Z_n$ and $Z_n/G(n)$ is a central factor of $H^*$.

By construction, the system of finite groups $\{H^*/N_n, \pi_{n,m}\mid n,m\geq1\}$, where
\[
\pi_{n,m}: \frac{H^*}{N_n} \to \frac{H^*}{N_m}, \qquad \pi_{n,m}(hN_n)=hN_m, \qquad m \le n,
\]
forms an inverse system and
\[
\left( \frac{H^*}{N_n} \right)'= \frac{(H^*)'N_n}{N_n}=\frac{LN_n}{N_n}=\frac{LG(n)N_n}{N_n}=
\frac{GN_n}{N_n} \cong \frac{G}{G\cap N_n} =\frac{G}{G(n)}.
\]
Set $K_n=H^*/N_n$ for every $n\geq1$, and $K:=\varprojlim_{n\geq1}K_n$.
Then $K$ is a profinite group, and since $K_n'\cong G/G(n)$ for every $n\geq1$, Proposition~\ref{prop:profiniteintegral} implies that $G\cong \overline{K'}$ --- recall that $G=\varprojlim_nG/G(n)$ (cf. Remark~\ref{rem:nikolov-segal}).
Therefore, $K$ is a profinite integral of $G$. 
Observe that the definition of $K$ yields a homomorphism of groups $\phi\colon H^*\to K$ with kernel $\bigcap_{n\geq1}N_n$.

Since $$[H^*:GN_n][GN_n:N_n]=[H^*:N_n]<\infty\quad\text{for every }n\geq1,$$ and $g_1,\ldots,g_s$ are the topological generators of the
profinite group $G$, then we have that $g_1N_n,\ldots,g_sN_n$ generate the finite group
$GN_n/N_n$. Moreover, if $t_1,\ldots, t_r$ are the abstract generators of the abstract group $T$, then the cosets $t_1 GN_n, \ldots, t_r GN_n$ generate the finite group $H^*/GN_n$.
Therefore the cosets
$$t_1 N_n,\ldots,t_r N_n, g_1N_n,\ldots, g_s N_n$$ generate the finite group $H^*/N_n$ for every $n\geq1$. 
Consequently, \cite[Lemma~2.4.1]{RibZal:book} implies that $K$ is (topologically) generated by the elements
\[
\phi(t_1) ,\:\ldots,\: \phi(t_r),\: \phi(g_1),\:\ldots,\: \phi(g_s),\]
namely, $K$ is a finitely generated profinite group.
In particular, $\overline{K'}=K'$ by Remark~\ref{rem:nikolov-segal}, and thus $K$ is both a profinite and an abstract integral of $G$. \qed
\end{proof}

The obvious generalisation of these two theorems would assert that if a
profinite group has an integral, then it has a profinite integral.
But this is false, as we show in Section \ref{sec:products}.

\medskip

It is natural to ask whether the finitely generated profinite group $G$ has a
profinite integral if and only if $G/G(n)$ is integrable for all $n\ge1$.
We can answer this question in the case when $Z(G)=1$. We begin with a simple
observation. 

\begin{lemma}\label{lemprof1} Let  $N$ be a characteristic subgroup of the group $G$ with $Z(G)\le N$. Then for every integral $H$ of $G$ there exists a uniquely defined section $H_1$ of $\Aut (G)$ such that $H_1$ is an integral of $G/N$ 
and a homomorphic image of $H$.
\end{lemma}

\begin{proof} Let 
 $K = C_H(G)$; then $H/K$ is (isomorphic to) a subgroup of $\Aut(G)$. Since $K\cap G = Z(G) \le N \trianglelefteq H$, we have
\[ \left(\frac{H}{NK}\right)' = \frac{GK}{NK} = \frac{G}{G\cap NK} = \frac{G}{N(G\cap K)} = \frac{G}{N};\]
thus $H_1 = H/NK$ is an integral of $G/N$. \qed
\end{proof}

\begin{theorem}\label{teoprof1}
Let $G$ be a finitely generated profinite group with $Z(G)=1$. Then the following are equivalent, 
\begin{enumerate}
\item\label{111} $G/G(n)$ is integrable for every $n\ge 1$;
\item\label{112} $G$ has a profinite integral, i.e., there exists a profinite group $K$ such that $G=\overline{K'}$.
\end{enumerate}
\label{t:cofz1}
\end{theorem}

\begin{proof} Since every $G(n)$ is a characteristic subgroup of $G$, (\ref{112}) implies that $G/G(n)=\overline{(K/G(n))'}$.
Since $G$ is finitely generated, $G(n)$ is an open subgroup of $G$ for every $n$ (cf. Remark~\ref{rem:nikolov-segal}), and thus $G/G(n)$ is a finite subgroup of the profinite group $K/G(n)$.
Hence, also $(K/G(n))'$ is a finite subgroup of $K/G(n)$, and thus $$\overline{(K/G(n))'}=(K/G(n))'$$ --- observe that every finite subgroup of a profinite group is closed, as profinite groups are totally disconnected (cf. \cite[Theorem~2.1.3:(b)]{RibZal:book}). Therefore, $G/G(n)={(K/G(n))'}$, and this shows the implication (\ref{112})~$\Rightarrow$~(\ref{111}). So we proceed in proving (\ref{111})~$\Rightarrow$~(\ref{112}).

\smallskip
Thus, let $G$ be a finitely generated profinite group, and suppose that $G/G(n)$ is integrable for every $n\ge 1$. 
For every $n\ge 1$, we write $Q_n = G/G(n)$ and denote by $\pi_n: G \to Q_n$ the natural projection.

Given an index $i\ge 1$ suppose that $Z(Q_{i+n}) \not\le \pi_{i+n}(G(i))$ for every $n\ge 0$. Then, as $G/G(i)$ is finite, there exists $x\in G$, with $1\ne xG(i)\in Z(Q_i)$ such that 
\[ \pi_{i+n}(x) \in Z(Q_{i+n})\]
for infinitely many $n\ge 0$.

Then  $[x, G] \le \bigcap_{n\ge 0}G(i+n) = 1$ and so $x$ is a non-trivial central element of $G$, which is a contradiction. Therefore, for every $i\ge 1$ there exists $i^{\ast}\ge i+1$ such that $G(i)/G(j) \ge Z(Q_j)$ for every $j \ge i^{\ast}$, and we may thus select an infinite subset $I = \{i_1, i_2, \ldots\}$ of positive integers such that $i_1 < i_ 2 < \ldots$ and 
\[ \pi_j(G(i_n)) \ge Z(Q_{i_m})\]
for every $i_n,i_m\in I$ with $n < m$.
For each $n\ge 1$, we write $G_n = Q_{i_n}$; so that $G_n$ is a quotient of $G_{n+1}$ modulo a characteristic subgroup; we also set
\[ \mathfrak{I}(G_n) = \{ H\mid H\ \mathrm{is\ a\ section\ of} \Aut(G_{n+1})\ \mathrm{and}\ H'\cong G_n\}. \]
By Lemma \ref{lemprof1} and the fact that $G_{n+1}$ is integrable, $\mathfrak I(G_n)$ is not empty, and finite; moreover, for every $Y \in \mathfrak{I}(G_{n+1})$ there are a uniquely defined $Y^{\ast}\in \mathfrak{I}(G_n)$ and a surjective homomorphism $Y \to Y^{\ast}$.\\
For every $n\ge 1$, and every pair $(H_n, H_{n+1}) \in  \mathfrak{I}(G_n)\times  \mathfrak{I}(G_{n+1})$, we then write an arrow
$H_n \to H_{n+1}$ if $H_n = H_{n+1}^{\ast}$.
This gives rise to an infinite locally finite directed tree which, by K\"onig's Lemma, has an infinite path
\[ H_ 1 \longrightarrow H_2 \longrightarrow \ldots \longrightarrow H_{n} \longrightarrow \ldots \]
Then (reversing the arrows) for each $n\ge 1$, there exists a surjective homomorphism $H_{n+1}\to H_n$, and by taking compositions we have, for every $1\le n < m$, a surjective homomorphism $\psi_{m,n}\colon H_m \to H_n$.
Let $$K:=\varprojlim_{n\geq1}H_n$$ be the profinite group associated to the inverse system thus defined. 
Since $G_n=G/G(i_n)\cong H_n'$ for every $n\geq1$, and $$G=\varprojlim_{n\geq1} G_{n}=\varprojlim_{n\geq 1}G/G(i_n)$$
(cf. Remark~\ref{rem:nikolov-segal}), Proposition~\ref{prop:profiniteintegral} implies that $G\cong \overline{K'}$, and thus $K$ is a profinite integral of $G$. \qed
\end{proof}

\begin{example}
\label{ex:finite-products}
 (finite groups) For $n\ge 2$, let $X_n$ be an elementary abelian $2$-group of order $2^n$; the number of maximal subgroups of $X_n$ is $\nu(n) = 2^n-1$. Let $M_1, M_2, \ldots  M_{\nu(n)}$ be the distinct maximal subgroups of $X_n$; let $p_1, p_2, \ldots, p_{\nu(n)}$ be distinct odd primes, and for each $i=1, \ldots, \nu(n)$, let $\langle a_i\rangle$ be a cyclic group of order $p_i$. We let $X_n$ act on the cyclic group $$ A_n= \langle a_1\rangle\times \cdots \times \langle a_{\nu(n)}\rangle$$ by setting
$C_{X_n}(a_i) = M_i$ and $X_n/M_i$ act as the inversion map on $\langle a_i\rangle$, for every $i =1, 2,\ldots, \nu(n)$.  We now consider the semidirect product $H_n = A_n\sdir X_n$. 
 Then $H_n$ is an integral of the cyclic group $A_n$, and $H_n/A_n$ is 
$n$-generated.
 
But for every subgroup $R$ with $A_n \le R < H_n$ we have $C_{A_n}(R) \ne 1$, and so $R' < A_n$. Thus, $H_n$ is a minimal integral of $A_n$ and, in particular, does not contain any integral of $A_n$ with less than $n$ generators.
\end{example}

\begin{example} (profinite groups) We first partition the set $\mathcal D$ of all odd primes in a countable union of disjoint sets 
$\mathcal D = \bigcup_{n\ge 2} \mathcal{D}_n$, 
with $|\mathcal{D}_n| = 2^n-1$ for every $n\ge 2$. Then, for every $n\ge 2$, we consider the group $H_n$ as constructed in Example~\ref{ex:finite-products} --- with $\{p_{n,1}, \ldots , p_{n,\nu(n)}\} = \mathcal{D}_n$ and $A_n=\langle a_{n,1},\ldots,a_{n,\nu(n)}\rangle$ ---, and set $H = \Car_{n\ge 2}H_n$.
Then $H$, endowed with the product topology, is a profinite group.
In particular, one has
\[
 H=\varprojlim_{n\geq2} \left(\left(\prod_{i=2}^nA_i\right)\rtimes \left(\prod_{i=2}^n X_i\right)\right) 
=A\rtimes X,
\]
where the epimorphisms $\prod_{i=2}^nA_i\twoheadrightarrow\prod_{i=2}^mA_m$ and
$\prod_{i=2}^nX_i\twoheadrightarrow\prod_{i=2}^mX_m$, for $n\geq m$, are the canonical projections.
Observe that $X:=\varprojlim_n(\prod_{i=1}^n X_i)$ is a pro-2-Sylow subgroup of $H$; while
\[
 A:=\varprojlim_{n\geq2}\left(\prod_{i=2}^nA_i\right)\simeq
 \left\{\:(k_p)_{p\in\mathcal{D}}\in\Car_{p\in\mathcal{D}}\mathbb{Z}/p\mathbb{Z}\:\right\},
\]
whose cardinality is a supernatural number prime to 2.
Moreover, since
\[
 \widehat{\mathbb{Z}}=\left\{(k_n)_{n\geq2}\in\Car_{n\geq2}\mathbb{Z}/n\mathbb{Z}\:\mid\:k_n\equiv k_m\bmod m\text{ for }m\mid n\right\},
\]
one has an epimorphism of profinite groups $\phi\colon \widehat{\mathbb{Z}}\to A$, so that $A$ is pro-cyclic, generated (as a profinite group) by the ``diagonal'' element 
\[
 \bar a:=\left(a_{n,i}\right)_{n\geq 2,\:1\leq i\leq \nu(n)}=\phi(1).
\]
Since $A_n=H_n'$ for every $n\geq2$ (cf. Example~\ref{ex:finite-products}), one has $\overline{H'}=A$ by Proposition~\ref{prop:profiniteintegral}.

On the other hand, we observe that $A$ does not coincide with $H'$, the abstract derived group of $H$. To see that, observe first that $H'$ is the subgroup generated by all commutators $[b,x]$ with $b\in A$, $x\in X$.
We claim that $\bar{a}$ cannot be expressed as a product of a finite number of such commutators. 
Indeed, set $h=[b_1,x_1]\cdots[b_r,x_r]$ for some $b_1,\ldots,b_r\in A$ and $x_1,\ldots,x_r\in X$, and $r\geq1$.
Pick $n$ such that $r<\nu(n)$, and let $\pi_n\colon A\to A_n$ and $\tau_n\colon X\to X_n$ denote the projections onto the $n$-th component.
Then 
\[\pi_n(h)=[\pi_n(b_1),\tau_n(x_1)]\cdots[\pi_n(b_r),\tau_n(x_r)]=a_{n,i_1}^{s_{i_1}}\cdots a_{n,i_r}^{s_{i_r}},\]
for some $1\leq i_j\leq\nu(n)$ and $1\leq s_{i_j}\leq p_{n,i_j}$ for every $j=1,\ldots,r$.
Since $\pi_n(\bar a)=a_{n,1}\cdots a_{n,\nu(n)}$, and $r<\nu(n)$, the elements $a_{n,i_1},\ldots, a_{n,i_r}$ are less than $\nu(n)$, and hence they are not enough to generate $\pi_n(\bar a)$, so that $h\neq \bar a$.
Therefore, $\bar a$ does not belong to the abstract derived group $H'$. 
\end{example}


\subsection{Products}
\label{sec:products}
 
In \cite{accm}, we asked whether the group $D_8\times D_8$ has an integral.
Here, we  answer this negatively, in a strong form: no finite direct power of
a non-abelian dihedral group has an integral.

\begin{prop}\label{procar1} 
Let $n\ge 3$ and let $G$ be a normal subgroup of the group $H$, with
$G\cong(D_{2n})^m$ (the direct product of $m$ copies of the dihedral group
$D_{2n}$). Then 
\[G\cap H'< G.\]
\end{prop}

\begin{proof}
For $i=1,\ldots m$, set
\[G_i = \langle y_i, x_i\mid y_i^n = x_i^2 = 1,\, y_i^{x_i} = y^{-1}\rangle,\]
and $G = G_1\times \ldots \times G_m$. Let $A = \langle y_1, \ldots, y_m\rangle$; then $A$ is a characteristic abelian subgroup of $G$, it is homocyclic of type $n^m$, and $C_G(A) = A$.

Let $K = C_H(A)$; then $K\unlhd H$ and $K\cap G = A$. Moreover, $H/K$ embeds in
$\Aut (A)$,
which is isomorphic to the group $\GL(m, \mathbb{Z}/n\mathbb{Z})$ of
all invertible $m\times m$ matrices over the ring $\mathbb{Z}/n\mathbb{Z}$.

Let $\bar x_1$ be the image of $(x_1, 1,  \ldots, 1)K$ in
$\GL(m,\mathbb{Z}/n\mathbb{Z})$. Then $\bar x_1$ acts on $A$ as the matrix
\[\left[\begin{array}{rrrr}
-1 & 0 & \cdots & 0\\
0 & 1 & \cdots & 0\\
\cdot & \cdot &\cdots & \cdot \\
0 & 0 & \cdots & 1\\
\end{array}\right],\]
and in particular $\det (\bar x_1) = -1$. This means that  $\bar x_i$ does not
belong to the derived subgroup of $\GL(m,\mathrm{Z}/n\mathrm{Z})$ and so it
cannot possibly belong to the derived group $(H/K)'  = H'K/K$.
Thus, $(x_1, 1,  \ldots, 1) \in G\setminus H'$. \qed
\end{proof}

Observe that, in the previous Proposition,  the group $H$ need not be finite.

\begin{cor} For every $n\ge 3$ and $m\ge 1$ the direct sum $(D_{2n})^m$
does not have an integral.
\end{cor}

Now we use the above result to construct a profinite group which has an 
integral, but does not have a profinite integral, or even a residually finite
integral. Our group is
\[G = \Car_{n\in \mathbb{Z}}G_n\]
where, for every $n\in \mathbb{Z}$,
$G_n\cong D_8 = \langle y_n, x_n\mid y_n^4 = x_n^2 =1,\, y_n^{x_n} = y_n^{-1}\rangle$, and $\Car$ denotes the unrestricted Cartesian product --- so, $G$ is profinite by Remark~\ref{rem:productsZ are profinite} below.
We  identify $G_n$ with the $n$-th coordinate subgroup of $G$, while for a
generic element of $G$, we write $\bar g = (g_n)_{n\in \mathbb{Z}}$, with
$g_n\in G_n$. We also set $u_n = y_n^2$, for every $n\in \mathbb{Z}$; thus
\[ Z(G) = G' = \Car_{n\in \mathbb{Z}}\langle u_n\rangle.\]

\begin{remark}\label{rem:productsZ are profinite}
Let $H$ be a finite group, and set $G=\Car_{n\in \mathbb{Z}}G_n$, where $G_n\cong H$ for every $n\in\mathbb{Z}$, and consider every $G_n$ as a discrete group.
Then $G$, endowed with the product topology, is a profinite group --- namely, every open subset $U$ of $G$ has the shape
\[U=\left(\Car_{n\in I}U_n\right)\times \left(\Car_{n\in \mathbb{Z}\smallsetminus I}G_n\right)\]
for some finite subset $I\subset \mathbb{Z}$, and some subset $U_n\subseteq G_n$ for every $n\in I$.
One may see $G$ as the projective limit of a directed system of finite groups as follows.
Write $\mathbb{Z}=\{i_1,i_2,\ldots,i_n,\ldots\}$, and for every $n\geq1$ set $P_n=\prod_{k=1}^n G_{i_k}$, endowed with the canonical projections $\varphi_{n,m}\colon P_n\to P_m$, with $\ker(\varphi_{n,m})=\prod_{k=m+1}^nG_{i_k}$, for every $n> m$.
Then $\{P_n,\varphi_{n,m}\}$ makes up a directed system of finite groups, and $G=\varprojlim_{n\geq1}P_n$.
\end{remark}

\begin{lemma}\label{lemcar1}
The unrestricted wreath product $D_8\wp \mathbb{Z}$ is an integral  of $G$.
\end{lemma}

\begin{proof}
By a result of Peter Neumann \cite[Corollary 5.3]{PN}, the derived
subgroup of the unrestricted wreath product $ D_8\wp \mathbb{Z}$ is exactly
the base group, which by definition is isomorphic to $G$. \qed
\end{proof}

\begin{prop}\label{procar2}
No integral of $G$ is residually finite.
\end{prop}

\begin{proof} Let
\[\mathcal L = \{ y_n^{\pm 1}\bar u\mid n\in \mathbb{Z},\, \bar u \in Z(G)\};\]
then easy considerations show that $\mathcal L$ is the set of all elements
$\bar g\in G$ such that $|\bar g|=4$ and $\langle \bar g\rangle \unlhd G$.

\smallskip

Suppose that the group $H$ is an integral of $G$. Then, for every
$n\in \mathbb{Z}$ and $h\in H$, $y_n^h\in \mathcal L$, that is,
$y_n^h = y_j^{\pm 1}\bar u$, for some $j\in \mathbb{Z}$ and $\bar u\in Z(G)$.
Consequently,
\[ u_n^h = (y_n^h)^2  = u_j.\]
This proves that, by conjugation, $H$ acts as a group of permutations on the
set $X =\{ u_n\mid n\in \mathbb{Z}\}$.

\smallskip

Assume now, for a contradiction, that $H$ is residually finite. Let $N$ be a
normal subgroup of finite index in $H$ such that $G_0\cap N = 1$, and
$M =  G\cap N$.

Then, for every $h\in N$, $y_0^{-1}y_0^h = [y_0,h]\in M$, that is,
$y_0^h = y_0\bar g$, with $\bar g \in M$. On the other hand, as observed before,
$y_0^h = y_j^{\pm 1}\bar u$, for some $j\in \mathbb{Z}$ and $\bar u\in Z(G)$.
Suppose that $j\ne 0$.
Then, since $M$ is normal in $G$, it contains
\[[\bar g, x_0] = (y_0^{-1}y_0^h)^{-1}(y_0^{-1}y_0^h)^{x_0}= (y_0^{-1}y_j\bar u)^{-1}(y_0^{-1}y_j\bar u)^{x_0}= y_0y_j^{-1}\bar u\cdot y_0y_j\bar u = u_0,\]
which is a contradiction. Thus, $y_0^h = y_0^{\pm 1}\bar u$, and
$u_0^h = (y_0^h)^2 = u_0$. Therefore, $N\le C_H(u_0)$; since $|H:N|$ is finite,
the $H$-orbit of $u_0$ by conjugation is finite.

Let $I$ be the finite subset of $\mathbb{Z}$ such that
$u_0^H = \{ u_i\mid i\in I\}$; then write $D = \Dir_{i\in I}G_i$ and
$Z_{\ast} = \Car_{j\in \mathbb{Z}\setminus I}\langle u_j\rangle$. 
As $Z_{\ast}$ is central in $G$, we have  $Z_{\ast}\unlhd G$. We claim that $D_{\ast} = DZ_{\ast} = DZ(G)$ is normal in
$H$.

For every $i\in I$ and $h\in H$ we have $y_i^h = y_t^{\pm 1}\bar u$, for some
$t\in \mathbb{Z}$ and $\bar u \in Z(G)$; as $u_i^h = (y_i^h)^2 = u_t$ belongs
to $u_i^H = u_0^H$, we have $t\in I$, whence $y_i^h\in DZ(G) = D_{\ast}$.
Consider now $x_i$ ($i\in I$) and $y_j$ with $j\in \mathbb{Z}\setminus I$;
let $h\in H$, then there exists $k\in \mathbb{Z}\setminus I$ such that
\[ [y_j, x_i^h]  = [y_j^{h^{-1}}, x_i]^h = [ y_k\bar u, x_i]^h = [ y_k, x_i]^h =1.\]
Therefore, if $Y = \Car_{j\in \mathbb{Z}\setminus I}\langle y_j\rangle$,
we have $x_i^h\in C_G(Y) = D\times Y$, and since $|x_i^h|=2$,
\[ x_i^h \in D\times Y^2 = D\times Z_{\ast} = D_{\ast}.\]
As $D = \langle y_i, x_i\mid i\in I\rangle$, we have proved that
$D^h\subseteq D_{\ast}$ for every $h\in H$, and consequently we have
$D_{\ast}\unlhd H$. Now,
\[ \frac{D_{\ast}}{Z_{\ast}} = \frac{DZ_{\ast}}{Z_{\ast}} \cong \frac{D}{D\cap Z_{\ast}} = D\]
is the direct sum of $|I|$ copies of the dihedral group $D_8$, and so, by Proposition \ref{procar1},
$D_{\ast}/Z_{\ast}$ is not contained in the derived group of
$H/Z_{\ast}$. Thus, $D_{\ast} \not\le H' = G$, which is the final
contradiction. \qed
\end{proof}

\subsection{Cartesian products and periodic integrals}
\label{sec:cartesian}

By \cite[Corollary 5.3]{PN}, we see that the unrestricted
wreath product of $S_3$ with $\mathbb{Z}$ is an integral of the Cartesian
product, thus we investigate periodic integrals of profinite groups in this section.

\begin{prop}\label{proper1} Let $G_n\cong S_3$ for every $n\in \Z$.
Then the group $G = \Car_{n\in \Z}G_n$ does not have periodic integrals.
\end{prop}
\begin{proof}
For each $n\in \Z$, let $G_n= \langle y_n,x_n\mid y_n^3=x_n^2 =1,\, y_n^{x_n} = y^{-1}\rangle $. We  identify $G_n$ with the $n$-th coordinate subgroup of $G$, while for a generic element of $G$, we write $\bar g = (g_n)_{n\in \Z}$, with $g_n\in G_n$. Let $\mathcal S = \{ \langle y_n\rangle \mid n\in \Z\}$; then $\mathcal S$ is the set of all normal subgroups of $G$ of order $3$, thus, in particular, every automorphism of $G$ permutes the elements of $\mathcal S$. 

\smallskip
Let the group $H$ be an integral of $G$,
and assume, by contradiction, that $H$ is periodic. Now, since it is contained in  $G = H'$, $x_0$ is the product of a finite number of commutators in $H$, so there exists $G < N\le H$ with $x_0\in N'$ and $N/G$ finitely generated, and, because $H/G$ is periodic abelian, $N/G$ is finite. As $\langle y_0\rangle \in \mathcal S$, it follows that the $N$-conjugation orbit of $\langle y_0\rangle$ is finite. Let $I$ be the finite subset of $\Z$ such that $\{\langle y_i\rangle\mid i\in I\}$ is the $N$-orbit of $\langle y_0\rangle$, and set $D = \Dir_{i\in I}G_i$.

\smallskip
Let $C = \Car_{j\in \Z\setminus I}\langle y_j\rangle$ so that
$C\times D'= \Car_{n\in \Z}\langle y_n\rangle = G'=H''$.
We show that $D_{\ast} := DC = D\times C$ is normal in $N$. We have just observed that $CD'$ is normal in $H$; thus, consider $x_i$ with $i\in I$, and $g\in N$. Now, for every $j\in \Z\setminus I$, $\langle y_j^{g^{-1}}\rangle = \langle y_k\rangle$ for some $k\in \Z\setminus I$, whence
\[y_j^{x_i^g} = (y_j^{g^{-1}})^{x_ig} = (y_j^{g^{-1}})^{g}=  y_j.\]
Therefore, $x_i^g \in C_G(C) = D\times C = D_{\ast}$.  Moreover, if $h\in N$ is such that $\langle y_i\rangle = \langle y_0\rangle^h$, that is $y_0^h = y_i^{\epsilon}$ with $\epsilon\in \{1,-1\}$, then 
\[ (y_i^{\epsilon})^{x_i} = y_i^{-\epsilon} =  (y_0^{-1})^h = (y_0^h)^{x_0^h} =  (y_i^{\epsilon})^{x_0^h},\]
showing that $x_0^h \in x_iC_{D_{\ast}}(y_i)$; similarly, $x_0^h\in \bigcap_{i\ne k\in I}C_{D_{\ast}}(y_k) = G_iG'$, 
and therefore $ x_0^h \in x_iG'$, yielding in particular $x_i \in N'$,
since $G' \le N'$ and $x_0^h \in N'$. Hence $x_i^g \in N' \cap D_{\ast}$ for every $g \in N$.

Since $D_{\ast}$ is generated by $G' \le N'$ and $\{x_i\}_{i \in I} \subseteq N'$ and that all of their $N$-conjugates still live in $D_{\ast}$, then $D_{\ast} \le N'$ is normal in $N$. 
Finally,
\[ \frac{D_{\ast}}{C} = \frac{DC}{C} \cong  \frac{D}{D\cap C} = D\]
is the direct sum of $|I|$ copies of the dihedral group $S_3$, and so by Proposition~\ref{procar1}
$D_{\ast}/C$ is not contained in the derived group of $N/C$, a contradiction to the fact that $D_{\ast} \le N'$. \qed
\end{proof}


\section{Questions}
\label{sec:questions}

We begin with a solution to Question 10.1 in our previous paper~\cite{accm}
by Efthymios Sofos. We are grateful to him for permission to publish it here.

\begin{theorem}[Sofos]
The number of integers $n$ with
$1<n<x$ for which every group of order $n$ is integrable is asymptotically
\[\mathrm{e}^{-\gamma} \frac{x }{ \log \log \log x },\]
where $\gamma$ is the Euler--Mascheroni constant.
\end{theorem}

\begin{proof}\textbf{(outline)}
This follows by a modification of the proof of the result of
Erd\H{o}s~\cite{erdos} for the number of integers $n$ for which every group
of order $n$ is cyclic. This
can be found as Theorem 11.23 in the book \cite{MV} (which we follow closely).
One has to replace property (i) by ``$n$ is cube-free''
and leave property (ii) as is. Then define $A_p(x)$ as the number of
``integrable'' $n\le x$ such that the least prime divisor of $n$ is $p$.
It is shown in pages 387 and 388 that 
\[\sum_{p\le\log\log x} A_p(x)=O(x (\log\log\log x)^{-2}),\]
but in fact only  property (ii) is used for this. So the same proof 
holds for our case as well.

The rest of the proof needs only small modification: replace
\begin{quote}
If $n$ does not satisfy (i), there is a prime $p$ with $p^2\mid n$. 
The number of such $n\le x$ is at most $\lfloor x/p^2\rfloor\le x/p^2$. 
Hence the total number of $n$ in $\Phi(x,y)$ for which (i) fails is not more
than $x\sum_{p>y}p^{-2} \ll x/(y\log y)$.
\end{quote}
by
\begin{quote}
If $n$ does not satisfy (i),  there is a prime $p$ with $p^3\mid n$. The
number of such $n\le x$ is at most $\lfloor x/p^3\rfloor\le x/p^3$. 
Hence the total number of $n$ in $\Phi(x,y)$ for which (i) fails is not more
than $x\sum_{p>y}p^{-3}\le\sum_{p>y}p^{-2}  \ll x/ (y\log y)$. \qed
\end{quote}
\end{proof}

Now we turn to some further open questions arising from this paper.

\begin{question} (Section~\ref{s:bounding})
Find a bound, or a procedure for calculating one, for the order of
the integral of an integrable finite group of order~$n$.
\end{question}

\begin{question} (Section~\ref{sec:necessary})
Find classes of groups $G$ for which the condition $\Inn(G)\le\Aut(G)'$ is
sufficient for integrability. We note that this is true in two extreme cases,
abelian groups and perfect groups.
\end{question}

\begin{question} (Section~\ref{sec:p-integrals})
Theorem~\ref{t:not-p-int}(b) shows that a non-abelian $p$-group whose derived
group has index~$p^2$ is not $p$-integrable. Is there a non-abelian
$p$-integrable $p$-group whose Frattini subgroup has index $p^2$ (that is, one
which is $2$-generated)?
\end{question}

\begin{question} (Section~\ref{sec:integrabelian})
Is the following true? The finite or countable abelian $2$-group $G$ is
finitely integrable if and only if it has subgroups $A,B,F$ such that
\[G\cong A\times A\times A\times B\times B\times F,\]
where $F$ is the direct product of a divisible group and a finite group.

In particular, is it true that a direct sum of finite cyclic groups is
finitely integrable if and only if the set of natural numbers $n$ for which
the cyclic group $C_{2^n}$ has multiplicity $1$ in the product is finite?
\end{question}

\begin{question} (Section~\ref{sec:integrabelian})
Is it true that integrability of a reduced $p$-group of arbitrary cardinality
is determined by its Ulm--Kaplansky invariants?
\end{question}

\begin{question} (Section~\ref{sec:varieties})
Let $\mathbf{V}$ be a finitely based variety of groups. We know that the class
of all integrals of groups in $\mathbf{V}$ is a variety. Is it finitely based?
\end{question}

\begin{question} (Section~\ref{sec:varieties})
\begin{enumerate}
\item Is it true that all groups of exponent $p$ and class at most $p-1$ are
integrable?
\item is it true that, if $p$ and $q$ are primes with $p\nmid q-1$, then
every group in the variety $\mathbf{A}_q\mathbf{A}_p$ is integrable?
\end{enumerate}
\end{question}

\begin{question} (Section~\ref{sec:profinite}) 
Let $G$ be a finitely generated profinite group $G$ and
let $G(n)$ be as in Remark \ref{rem:nikolov-segal}.
Does $G$ have a (profinite) integral if and only if $G/G(n)$ has an integral 
for every $n\ge 1$? (Theorem~\ref{t:cofz1} shows that this is true if $Z(G)=1$.)
\end{question}

\begin{question} (Section~\ref{sec:profinite})
\begin{enumerate}
\item Does there exist a countable integrable locally finite group which does not have a periodic integral?
\item Is it true that every integrable finitely generated periodic group has a periodic (finitely generated) integral?
\item Which infinite periodic groups have a (periodic) integral? For example,
what about Grigorchuk's first group?
\end{enumerate}
\end{question}


\section*{Acknowledgments}

The authors would like to thank the referee for a careful reading which led to several
improvements of the exposition and proofs.

The first author was funded by national funds through the FCT - Funda\c{c}\~ao para a Ci\^encia e a Tecnologia, I.P., under the scope of the projects UIDB/00297/2020, UIDP/00297/2020 (Center for Mathematics and Applications) and 
\\ PTDC/MAT/PUR/31174/2017.

The first, second and fourth authors gratefully acknowledge the support of the 
Funda\c{c}\~ao para a Ci\^encia e a Tecnologia (CEMAT-Ci\^encias FCT projects UIDB/04621/2020 and UIDP/04621/2020).

The fourth and the fifth authors are members of the Gruppo Nazionale per le Strutture Algebriche, Geometriche e le loro Applicazioni (GNSAGA) of the Istituto Nazionale di Alta Matematica (INdAM), and the fourth author gratefully acknowledges the support of the Universit\`a degli Studi di Milano--Bicocca
(FA project ATE-2016-0045 ``Strutture Algebriche'').


\end{document}